\documentclass[11pt,reqno]{amsart}

\usepackage[paperwidth=8.26in,paperheight=11.69in,
left=1.125in, right=1.125in, bottom=1.25in]{geometry}
\usepackage[utf8]{inputenc}

\usepackage{amsmath,amssymb,amsthm}
\usepackage{enumerate,mathtools}
\usepackage{lscape}
\newtheorem{theorem}{Theorem}[section]
\theoremstyle{plain}
\newtheorem{corollary}[theorem]{Corollary}
\newtheorem{definition}[theorem]{Definition}
\newtheorem{proposition}[theorem]{Proposition}
\theoremstyle{remark}
\newtheorem{remark}{Remark}
\theoremstyle{example}
\newtheorem{example}[theorem]{Example}
\newtheorem{question}{Question}

\usepackage{cleveref}
\crefname{theorem}{Theorem}{Theorems}
\crefname{proposition}{Proposition}{Propositions}
\crefname{lemma}{Lemma}{Lemmas}
\crefname{corollary}{Corollary}{Corollaries}
\crefname{example}{Example}{Examples}
\crefname{definition}{Definition}{Definitions}
\crefname{remark}{Remark}{Remarks}
\crefname{question}{Question}{Questions}
\crefname{table}{Table}{Tables}
\crefname{enumi}{}{}
\crefname{enumii}{}{}
\crefname{enumiii}{}{}
\crefname{equation}{}{}

\numberwithin{equation}{section}

\newcommand{\C}{\mathbb{C}}

\DeclareMathOperator{\aut}{Aut}

\newcommand{\sff}{{\rm I\!I}}
\newcommand{\Wbar}{\overline{W}}
\newcommand{\Wba}{\overline{W}}
\newcommand{\Zbar}{\overline{Z}}

\newcommand{\kbar}{\bar{k}}

\newcommand{\zbar}{\bar{z}}

\newcommand{\wbar}{\bar{w}}
\newcommand{\ahlfors}{\mathcal{A}}
\newcommand{\grad}{\mathrm{grad}}

\newcommand{\Wm}{\mathcal{W}}
\DeclareMathOperator{\tf}{tf}
\DeclareMathOperator{\trace}{trace}
\DeclareMathOperator{\Sym}{Sym}

\newcommand{\abar}{\bar{\alpha}}
\newcommand{\bbar}{\bar{\beta}}

\title[The CR Ahlfors derivative and a new invariant for spherically equivalent]{The CR Ahlfors derivative and a new invariant for spherically equivalent CR maps}

\author{Bernhard Lamel}
\address{Texas A\&M University Qatar, Science Program, Education City, Doha, Qatar}
\email{bernhard.lamel@qatar.tamu.edu}

\author{Duong Ngoc Son}
\address{Fakultät für Mathematik, Universität Wien, Oskar-Morgenstern-Platz 1, 1090 Wien,\break Austria }
\email{son.duong@univie.ac.at}

\newcommand{\edit}[1]{{{#1}}}
\newcommand{\editb}[1]{{{#1}}}
\newcommand{\editc}[1]{{{#1}}}

\begin{document}

\date{October 30, 2020}
\subjclass[2000]{32V05, 32H35, 53A30}
\thanks{Duong Ngoc Son was supported by the Austrian Science Fund, FWF-Projekt M 2472-N35. Bernhard Lamel was supported by the Austrian Science Fund, FWF-Projekt I3472.}
\begin{abstract}
	We study a CR analogue of the Ahlfors derivative for conformal immersions of Stowe \cite{stowe2015ahlfors} that generalizes the CR Schwarzian derivative studied earlier by the second-named author \cite{son2018schwarzian}. This notion possesses several important properties similar to those of the conformal counterpart and provides a new invariant for \edit{spherically} equivalent CR maps from strictly pseudoconvex CR manifolds into a sphere. The invariant is computable and distinguishes many well-known sphere maps. In particular, it vanishes precisely when the map is spherically equivalent to the linear embedding of spheres.
\end{abstract}
\maketitle
\section{Introduction}
The main purpose of this paper is to extend the notion of \edit{the} CR Schwarzian derivative for CR diffeomorphisms \cite{son2018schwarzian} to the case of CR immersions.  For conformal immersions of Riemannian manifolds, the Ahlfors derivative of Stowe \cite{stowe2015ahlfors} generalizes the Schwarzian derivative of Osgood--Stowe \cite{osgood1992schwarzian} in a similar way and goes back to Ahlfors \cite{ahlfors1988cross}. Precisely, we shall define, for each CR immersion $f\colon (M,\theta) \to  (N,\eta)$ between pseudohermitian manifolds, a tensor denoted by $\ahlfors_{\theta}(f)$. This tensor reduces to the Schwarzian tensor introduced  in \cite{son2018schwarzian} for CR diffeomorphisms in the equidimensional case. We shall call this tensor the CR Ahlfors \edit{derivative (or tensor)}\editb{. We} refer the reader to Stowe's paper for further discussions regarding the history and motivations in the conformal case; however, it turns out that the CR setting has some special properties not present in the conformal setting, which we shall point out as we go.

The tensor which we are going to construct satisfies a ``chain rule'' described as follows: For a chain of CR immersions $(M,\theta) \xrightarrow{~F~} (N,\eta) \xrightarrow{~G~} (P,\zeta)$, it holds that
\begin{equation}\label{e:cr}
	\ahlfors(G\circ F)
	=
	\ahlfors(F) + F^{\ast} \ahlfors(G).
\end{equation} 
It was shown in \cite{son2018schwarzian} that if $(M,\theta)=(N,\eta)$ is the sphere with its standard pseudohermitian structure, then $\ahlfors(F)$ (which reduces to the CR Schwarzian derivative as already noted) vanishes identically if $F$ is a  CR automorphism of the sphere. Therefore, the chain rule \cref{e:cr} implies that $\ahlfors$ is an invariant for spherically equivalent CR maps into spheres of higher dimensions. This invariant property is a main motivation for us to extend the notion of \edit{the} CR Schwarzian derivative to the case of higher dimensional targets. We shall in fact apply the Ahlfors derivative to study equivalence of sphere maps, a problem which has been studied extensively; we can mention only several papers \cite{d1988proper,d1991polynomial,d2016homotopy} and refer the readers to numerous references therein.

To construct the CR Ahlfors derivative, we shall follow Stowe's construction for the conformal case. First, we define a notion of CR second fundamental form for the ``isopseudohermitian'' immersions and the $(1,0)$-mean curvature vector (this step was not needed in the equidimensional case). Precisely, let $(N,\eta)$ be \edit{a} pseudohermitian manifold and let $\iota \colon (M,\theta) \hookrightarrow (N, \eta)$ be a pseudohermitian submanifold of $N$. This means the standard inclusion $\iota$ is CR and $\theta = \iota^{\ast} \eta$.
We denote by  $\nabla$ and $\widetilde{\nabla}$  the Tanaka-Webster connections on $(M,\theta)$ and $(N, \eta)$, respectively, introduced by Tanaka \edit{and Webster} \cite{tanaka1975differential}. For any two vector fields $X, Y \in \Gamma(\C TM )$ extended to smooth sections $\widetilde{X}, \widetilde{Y}$ of $\C TN$, we define the pseudohermitian \textit{second fundamental form} by the Gauß formula, namely,
\begin{equation}\label{e:sffdef}
\sff(X, Y) = \sff_M^N(X, Y) 
:=
\widetilde{\nabla}_{\widetilde{X}} \widetilde{Y} - \nabla_XY,
\end{equation}
 This notion was \edit{previously} studied by many authors, see, e.g., \cite{webster1979rigidity} for the codimension one case and \cite{dragomir1995pseudohermitian,ebenfelt2004rigidity}  for the case of pseudohermitian immersions (i.e., when the Reeb field of $\eta$ is tangent to $\iota(M)$.) For our applications, we shall make no assumption on the Reeb field of the target.

Due to the presence of the torsion, $\sff$ is not \edit{necessarily} symmetric and thus we also consider the symmetrized second fundamental form, i.e.,
\begin{equation}
\Sym \sff (X,Y) = \tfrac{1}{2}\left(\sff(X,Y) + \sff(Y,X)\right).
\end{equation}
In most situations, we shall consider the second fundamental form $\sff$ as a tensor on the ``horizontal'' space $T^{0,1}M \oplus T^{1,0}M$ (the ``good directions'') where it behaves quite well. In particular, we define the $(1,0)$-mean curvature vector to be the trace of $\sff$ on the horizontal subspace:
\begin{equation}
	H :=
	\sum_{\alpha = 1}^{n} \sff (Z_{\abar}, Z_{\alpha})
\end{equation}
where $\{Z_{\alpha} \colon \alpha = 1,2,\dots ,n\}$ is an orthonormal frame of $T^{1,0}M$. The trace of $\Sym\sff$ is denoted by $\mu$, so that $\mu = \Re H$. Let us point out that the consideration here is similar to \cite{son2019semi} in which we consider the case of CR immersions into a Kähler manifold. Moreover, when the target is the standard sphere, the second fundamental form \eqref{e:sffdef} is closely related to the one for CR immersions into complex euclidean space.

Similar to \cite{stowe2015ahlfors}, we define the tensor $\nu = \nu_M^N$ as a symmetric real tensor on $T^{1,0}M \oplus T^{0,1}M$ via the formula
\begin{equation}
		\nu(X,Y) = 2\left\langle \Sym \sff (X,Y) , \mu \right\rangle - \langle X,Y\rangle |\mu|^2.
\end{equation}
Moreover, we define, for each smooth function $u$ on $M$,
\begin{align}
	\mathcal{H}_{\theta}(u)
	=
	\Sym \nabla \nabla u -\partial_b u \otimes \partial_b u - \bar{\partial}_b u \otimes \bar{\partial}_b u + \frac{1}{2}|\bar{\partial}_b u|^2 L_\theta.
\end{align}
\edit{Here, $L_{\theta}(Z,\Wba): = -i d\theta(Z,\Wba)$ ($Z,W \in T^{1,0}M$) is the Levi form.} We refer the reader to \cref{ss:changecontact}, in particular \eqref{e:dudecomp}, for the (standard) notation used here. 
We can now introduce the CR analogue of the Ahlfors derivative as follows.
\begin{definition}[cf. \cite{stowe2015ahlfors}]
	Let $(M^{2n+1},\theta)$ and $(N^{2d+1},\eta)$ be strictly pseudoconvex pseudohermitian manifolds and let $F\colon M\to N$ be a CR immersion.
	Let $u$ be the smooth function on $M$ such that $F^{\ast} \eta = e^{u} \theta$. We define the CR
	Ahlfors derivative \edit{(or CR Ahlfors tensor)} of $F$ to be
	\begin{equation}
	\ahlfors(F)
	:=
	\mathcal{H}_{\theta}(u) 
	+ F^{\ast}\left(\nu^{N}_{F(M)}\right)
	+\frac12 F^{\ast}(J_{\Theta} L_{\Theta})  - \frac12  J_{\theta} L_{\theta}.
	\end{equation}
	where $J_{\theta} = R_{\theta}/(n(n+1))$ and $J_{\Theta} = R_{\Theta}/(d(d+1))$ are the normalized Webster scalar curvatures on $M$ and $N$, respectively, and \edit{$L_{\Theta}$ and $L_{\theta}$ are the corresponding Levi forms.}
\end{definition}
As mentioned above, the Ahlfors tensor $\ahlfors$ generalizes  the CR Schwarzian tensor for CR diffeomorphisms in \cite{son2018schwarzian} in the same spirit the conformal Ahlfors generalizes the Schwarzian of Osgood--Stowe. We shall explain this in the next section.

As briefly discussed, our first result of the paper is the following chain rule.
\begin{theorem}\label{thm:crintro}
	For CR immersions $F\colon (M,\theta) \to (N,\eta)$ and $G\colon (N,\eta) \to (P,\zeta)$, we have
	\begin{equation}\label{e:chainrule}
	\ahlfors(G\circ F)
	=
	\ahlfors(F)
	+
	F^{\ast}\ahlfors (G).
	\end{equation}
\end{theorem}
We point out that although this theorem is analogous to Theorem 1 in \cite{stowe2015ahlfors}, the chain rule in CR case is in fact simpler than its conformal counterpart: the excess term $\epsilon$ does not appear in \edit{the} CR case.

Specializing this chain rule to the case of CR maps into \edit{the sphere}, we obtain a new tensorial invariant for spherically equivalent classes of such maps. This invariant property is a consequence of the fact that the Schwarzian tensor of a CR automorphism of the sphere with \edit{its} standard pseudohermitian structure vanishes identically \cite{son2018schwarzian}. Recall that two CR maps $F$ and $G$ from $M$ into $\mathbb{S}^{2N'+1}$ are said to be (left) spherically equivalent if there exists a CR automorphism $\phi$ of $\mathbb{S}^{2N' +1}$ such that $G = \phi \circ F$. When $M=\mathbb{S}^{2N+1}$ is also a sphere, we can use the CR automorphisms of $M$ to define a weaker version of spherical equivalence. Namely, we say that $F$ and $G$ are spherically equivalent if there exist CR automorphisms $\gamma$ (of $\mathbb{S}^{2N+1}$) and $\phi$ (of $\mathbb{S}^{2N'+1})$ such that  $G\circ \gamma = \phi \circ F$. 

In the following, the unit spheres are always equipped with their standard pseudohermitian structures.
\begin{corollary}\label{cor:se}
	Let $(M,\theta)$ be a strictly pseudoconvex pseudohermitian manifold.
	\begin{enumerate}[(i)]
		\item Suppose that $F\colon M \to \mathbb{S}^{2N+1}$ is a CR immersion and $\phi \colon \mathbb{S}^{2N+1} \to \mathbb{S}^{2N'+1}$ ($N' \geq  N$) is a totally geodesic embedding, then
		\begin{equation}\label{e:invariant}
		\mathcal{A}(F) = \mathcal{A}(\phi \circ F).
		\end{equation}
		In particular, if $F$ and $G$ are left spherical equivalent CR maps from $M$ into $\mathbb{S}^{2N+1}$, then 
		\begin{equation} 
			\ahlfors(F) = \ahlfors(G).
		\end{equation} 
		\item Suppose that $G \colon (\mathbb{S}^{2n+1}, \Theta) \to (M,\theta)$ is a CR immersion and $\gamma \colon N \to \mathbb{S}^{2n+1}$ is a totally geodesic embedding, then
		\begin{equation}\label{e:invariant2}
		\gamma^{\ast}\ahlfors(G) = \ahlfors(G \circ \gamma).
		\end{equation}
	\end{enumerate}
\end{corollary}
\begin{remark}
	In Part (ii), if $N$ admits a totally geodesic embedding into a sphere, then it is \edit{necessarily} CR spherical. This follows from \cite{ebenfelt2004rigidity} for the case $\dim_{\mathbb{R}} N \geq 5$ and \cite{son2019semi} for the case $\dim_{\mathbb{R}} N = 3$.
\end{remark}

In view of \cref{cor:se}, an interesting question that arises is whether the Ahlfors derivative distinguishes the spherical equivalent classes of sphere maps. Although we can check that this is the case for maps between \edit{spheres} of ``low'' codimension, we do not know the answer to this question in the full generality. However, we prove that the CR Ahlfors distinguishes the totally geodesic CR map: an arbitrary CR map into the sphere with vanishing \edit{CR Ahlfors derivative} must be a totally geodesic embedding.

\begin{theorem}
	Let $F\colon (M,\theta) \to (\mathbb{S}^{2d+1}, \Theta)$ be a CR immersion. If $\ahlfors(F) = 0$, then $M$ is CR spherical and $F$ is spherically equivalent to the linear mapping.
\end{theorem}
It is \edit{also} natural to ask under which conditions the Ahlfors derivative is a nonzero functional multiple of the Levi metric. We shall discuss this question after analyzing several examples \edit{in} the last section; see \cref{q:1}.

The paper is organized as follows. In section 2, we study the geometry of the CR second fundamental form for CR immersions. In section 3, we prove \cref{thm:crintro}. In section 4, we study the maps with vanishing CR Ahlfors derivatives. We study the case when the source is of \edit{dimension three} in section 5. In section 6, we provide an explicit \edit{formula} for the Ahlfors derivative which is used to analyze various examples in section 7.

\noindent{\bf Acknowledgment.} The authors would like to thank an anonymous referee for very careful reading of the manuscript and pointing out many, many typographical errors that we were not aware of. 
 
\section{Immersions of CR manifolds and the second fundamental form}
\subsection{The second fundamental form}
Let $\iota \colon (M,\theta) \hookrightarrow (N, \eta)$ be a pseudohermitian submanifold, i.e., \edit{$\iota$ is CR} and $\theta = \iota^{\ast} \eta$, where $\iota$ is the inclusion. In this case, $\iota$ is ``isopseudohermitian'' in the sense of \cite{dragomir1995pseudohermitian}. This notion is more general than that of ``pseudohermitian immersions,'' as the latter requires that the Reeb field of $\eta$ is tangent to $M$ \edit{along $M$}. In the latter case, \edit{the pair} $(\theta, \eta)$ is admissible in the sense of \cite{ebenfelt2004rigidity}.

For any two vector fields $X, Y \in \Gamma(\C TM )$ extended to smooth sections $\widetilde{X}, \widetilde{Y}$ of $\C TN$, we recall that  the 
\textit{second fundamental form} is defined by
	$\sff(X, Y) =
	\widetilde{\nabla}_{\widetilde{X}} \widetilde{Y} - \nabla_XY$ (see \eqref{e:sffdef})
where $\widetilde{\nabla}$ and $\nabla$ is the Tanaka-Webster connection on $(N ,\eta)$ and $(M , \theta)$, respectively. We summarize the basic properties of $\sff$ as follows (cf. \cite{son2019semi} which treats a similar situation), 
where $T$ and $\widetilde{T} $ denotes the Reeb field of 
$(M,\theta)$ and $(N,\eta)$, respectively.
\begin{proposition}
The second fundamental form $\sff$ is well-defined, tensorial, and satisfies the following properties
for all $(1,0)$-vectors $Z$ and $W$:
\begin{align}
	\sff(\Zbar, \Wbar) & = \overline{\sff(Z,W)}, \label{e:ssf1}\\
	\sff(Z,\Wbar) & = \overline{\sff(\Zbar, W)},\label{e:ssf2}\\
	\sff(Z,W) & = \sff(W,Z), \label{e:ssf3} \\
	\sff(Z,\Wbar ) & = \sff(\Wbar , Z) - i \langle Z, \Wbar \rangle_{\theta} (T - \widetilde{T}), \label{e:ssf4} \\
	\sff(Z, T) & = \widetilde{\nabla}_Z (T - \widetilde{T}), \label{e:ssf5} \\
	\sff(T,Z) & = \widetilde{\nabla}_{T - \widetilde{T}} Z + [Z , T - \widetilde{T}] + \widetilde{\tau} Z - \tau Z. \label{e:ssf6}
\end{align}
Here $ \widetilde{\tau} Z : = \mathbb{T}_{\widetilde{\nabla}}(\widetilde{T}, Z)$ is the pseudohermitian torsion of $\widetilde{\nabla}$ and similarly for $\tau$. Moreover, $\sff$ is symmetric if and only if $\iota$ is pseudohermitian (i.e. $\widetilde{T} = \iota_{\ast} T$).
\end{proposition}
\begin{proof}
	That $\sff$ is well-defined and tensorial follows from standard arguments. Equations \cref{e:ssf1,e:ssf2} follow from the reality of the Tanaka-Webster connection. Equation \cref{e:ssf3} follows from the equation $\mathbb{T}_\nabla(Z,W) = 0$ for $(1,0)$-vectors $Z$ and $W$ \edit{on $M$ and similarly for $N$}. Proof of \cref{e:ssf4} uses the fact that $\mathbb{T}_{\nabla}(Z, \Wbar ) = i\langle Z,\Wbar \rangle T$. Precisely, by \cite{tanaka1975differential}
	\begin{equation}
		\nabla_Z \Wbar - \nabla_{\Wbar} Z - [Z, \Wbar] = \mathbb{T}(Z, \Wbar) = i\langle Z , \Wbar \rangle T,
	\end{equation}
	and similarly for $\widetilde{\nabla}$ and thus \cref{e:ssf4} follows.
	
	To prove \cref{e:ssf5}, observe that $\nabla T = 0$ and $\widetilde{\nabla}\widetilde{T} = 0$ \cite{tanaka1975differential}, and hence $\sff (Z,T) 
		=
		\widetilde{\nabla}_ZT - \nabla_ZT
		=
		\widetilde{\nabla}_Z(T - \widetilde{T})$, as desired. The proof of \cref{e:ssf6} also follows from direct calculations. We omit the details.
		
		Assume that $\widetilde{T}  = \iota_{\ast} T$, then $\iota$ is called a pseudohermitian immersion \cite{dragomir1995pseudohermitian} and the pair $(\theta, \eta)$ is said to be an admissible pair \cite{ebenfelt2004rigidity}. In this case, it follows from \cref{e:ssf5} that $\sff(Z,T) = 0$. On the other hand, from \cref{e:ssf6}, $\sff(T,Z) = \widetilde{\tau} Z - \tau Z$ and hence both sides vanish by type consideration. Similarly, from \edit{\cref{e:ssf4}}, $\sff(Z,\Wbar) = \sff(\Wbar,Z)$ and thus both sides vanish. Thus, $\sff$ is symmetric.
		
		Conversely, if $\sff$ is symmetric, then it follows from \cref{e:ssf4} that $\widetilde{T}  = \iota_{\ast} T$. The proof is complete.
\end{proof}
\begin{definition}
	Let $\iota \colon (M, \theta) \hookrightarrow (N, \eta)$ be a pseudohermitian submanifold, $\iota^{\ast} \eta = \theta$. The $(1,0)$-\emph{mean curvature} vector of $M$ in $N$ is the $(1,0)$-vector defined by
	\begin{equation}
	H:= \frac{1}{n}\sum_{\alpha = 1}^n \sff(Z_{\abar} , Z_{\alpha}),
	\end{equation}
	where $\{Z_{\alpha} \colon \alpha = 1,2,\dots , n\}$ is an orthonormal frame of $T^{1,0}M$ and $Z_{\abar} = \overline{Z}_{\alpha}$. 
	We also define $\mu = \mu_M^N$ to be 
	the trace of $\Sym\sff$, i.e. $\mu = \Re H$.
\end{definition}
Let $\iota \colon (M, \theta) \hookrightarrow (N, \eta)$ be a pseudohermitian submanifold, $\iota^{\ast} \eta = \theta$. Then $\iota_{\ast}$ sends $T^{1,0}M$ into $T^{1,0} N$. We can define the normal bundle $N^{1,0}M$ as a subundle of $T^{1,0} N$ as usual. Here the orthogonality only depends on the CR structure, but not on the pseudohermitian structure. We also define $N^{0,1}M$ similarly.

\begin{proposition}\label{prop:ttprime}
	Let $\iota \colon (M, \theta) \hookrightarrow (N, \eta)$ be a pseudohermitian submanifold, $\iota^{\ast} \eta = \theta$. If $T$ and $\widetilde{T}$ are the Reeb fields corresponding to $\theta$ and $\eta$, respectively, then for all tangent vectors $Z$, $W$ in $T^{1,0}M$,
	\begin{align}
		H - \overline{H} & = i(T - \widetilde{T}), \label{e:s1}\\
		\sff(Z,\Wbar) & = \langle Z , \Wbar \rangle \overline{H}, \label{e:s2}\\
		\widetilde{\tau} Z - \tau Z
			&=
		\sff(T,Z) - \sff(Z,T) = -i \widetilde{\nabla}_Z\overline{H}, \label{e:s3} \\
		\sff(T,Z) & = -i \widetilde{\nabla}_Z H.
	\end{align}
Moreover, $H \in N^{1,0}M$ and $\sff(Z,W) \in N^{1,0}M$.
\end{proposition}
\begin{proof} By direct calculation using \eqref{e:ssf4}, one has
\begin{align}
	 H
 	& =
 \frac{1}{n} \sum_{\alpha = 1}^n \sff(Z_{\abar}, Z_{\alpha}) \notag \\
	& =
 \frac{1}{n} \sum_{\alpha = 1}^n\left(\sff(Z_{\alpha} , Z_{\abar} )  + i \langle Z_{\alpha} , Z_{\abar}\rangle_{\theta}(T - \widetilde{T})\right) \notag \\
 	& =
 \overline{H} + i( T - \widetilde{T}).
\end{align}

\edit{Combining with \cref{e:ssf4}, we have} \edit{$\sff(Z,\Wbar) - \sff(\Wbar , Z) = -i \langle Z , \Wbar \rangle (T - \widetilde{T}) = \langle Z , \Wbar \rangle \left(\overline{H} - H\right)$}. Taking the (1,0) and (0,1) parts, we obtain \cref{e:s2}. 

To show that $H\in N^{1,0}M$, observe that $d\theta = d(\iota^{\ast}\eta) = \iota^{\ast}(d\eta)$. Thus, for every $X \in T^{1,0}M$,
\begin{equation}
	0 = d\theta(T,X) = d\eta(T,X).
\end{equation}
Therefore, if $X$ is tangent to $M$,
\begin{equation}
	\langle \overline{H}, X \rangle 
	=
	\langle \overline{H} - H , X \rangle 
	=
	-i \langle T - \widetilde{T} , X \rangle 
	=
	0.
\end{equation}
This implies that $\overline{H} \in N^{0,1}M$ and $H \in N^{1,0}M$. Finally, for $\Zbar, \Wbar \in T^{0,1}M$ and $X \in T^{1,0}M$, it holds that
\begin{equation}
	\langle \sff(\Zbar , \Wbar) , X \rangle
	=
	-\langle \Wbar , \sff(\Zbar, X) \rangle
	=
	-\langle \Zbar , X \rangle \langle \Wbar , H \rangle 
	=0.
\end{equation}
This implies that $\sff(\Zbar , \Wbar) \in N^{0,1}M$ and $\sff(Z,W) \in N^{1,0}M$. The proof is complete.
\end{proof}
\subsection{Change of contact forms}\label{ss:changecontact}
Let $\iota \colon (M, \theta) \hookrightarrow (N, \eta)$ be a pseudohermitian submanifold, $\iota^{\ast} \eta = \theta$. The total differential $du$ \editb{of a smooth function $u$} can be decomposed into $(1,0)$, $(0,1)$, and the transverse parts as follows:
\begin{equation}\label{e:dudecomp}
du = \partial_b u + \bar{\partial}_b u + (T^{\eta} u) \eta.
\end{equation}
This decomposition depends on the choice of pseudohermitian structure~$\eta$. We then define
\begin{equation}
\grad^{1,0} u 
=
u^{\gamma} Z_{\gamma},
\quad
\grad^{0,1} u
=
\overline{\grad^{1,0} \bar{u}}.
\end{equation}
If $\iota \colon M\hookrightarrow (N,\eta)$ and $\theta = \iota^{\ast} \eta$, then we have
\begin{equation}
\grad^{1,0}_N u = \grad^{1,0}_{N,M} u + (\grad^{1,0}_{N,M} u)^\perp,
\end{equation}
and similarly for $\grad^{0,1} u$. Here the orthogonal complements in $T^{1,0} N$ and $T^{0,1} N$ are defined using the Levi metric corresponding to any pseudohermitian structure on  $N$.

\begin{proposition}\label{prop:conch}
	Let $\iota \colon (M, \theta) \hookrightarrow (N, \eta)$ be a pseudohermitian submanifold, $\iota^{\ast} \eta = \theta$. Suppose that $P$ is the manifold $N$ with $\widetilde{\eta} = e^{u} \eta$. Put $\widetilde{\theta}: = e^{u\circ \iota} \theta$. Let $\widetilde{\sff}$ be the second fundamental form of the inclusion $\iota\colon (M,\widetilde{\theta}) \hookrightarrow (P,\widetilde{\eta})$. Then
	\begin{equation}\label{e:crsff}
	\widetilde{\sff}(Z, W) = \sff(Z, W), \quad Z, W \in T^{1,0}M,
	\end{equation}
	\begin{equation}\label{e:221}
	\widetilde{\sff}(\Zbar, W) = \sff(\Zbar , W) - \langle \Zbar , W\rangle_{\eta} (\grad^{1,0}_{N,M} u)^\perp, \quad \Zbar  \in T^{0,1}M, \ W \in T^{1,0}M,
	\end{equation}
	and
	\begin{equation}\label{e:224}
	e^{u}\mu_M^P
	=
	\mu_M^N -  \Re (\grad^{1,0}_{N,M} u)^\perp .
	\end{equation}
\end{proposition}
\begin{proof}
	The first two identities follow from Lee's formulas for the pseudoconformal change of the metrics \cite{lee1986fefferman}. Precisely, on $M$, we have
	\begin{equation}
	\widetilde{\nabla}_{Z}W = \nabla_ZW + Z(u) W + W(u) Z,
	\end{equation}
	and
	\begin{equation}\label{e:224a}
	\widetilde{\nabla}_{\Zbar}W = \nabla_{\Zbar}W  - \langle W,\Zbar\rangle_{\theta}\, \grad^{1,0}_N u.
	\end{equation}
	Similar formulas hold on $M$ \edit{and hence \cref{e:crsff,e:221} follow immediately. The last identity \cref{e:224} also follows by taking the trace of \cref{e:221} and its conjugate.}
\end{proof}
\begin{remark}
In view of \cref{e:crsff}, $\sff(Z,W)$, \edit{where $Z,W$ are $(1,0)$-vectors}, is called the CR second fundamental 
form of the CR immersion. It can be computed by any pair of pseudohermitian structures $\theta = \iota ^{\ast} \eta$, not necessary admissible. This notion has been extensively used in the study of the CR immersions \cite{dragomir1995pseudohermitian,ebenfelt2004rigidity,ebenfelt2014cr}.
\end{remark}
\subsection{The Gauß and Weingarten equations}

\begin{proposition}[Pseudohermitian Weingarten Equation] If $N$ is a section of $N^{1,0}M \oplus N^{0,1}M$, then
	\begin{equation}
		\langle\widetilde{\nabla}_ X N , Y \rangle 
		=
		-\langle N , \sff(X,Y) \rangle 
	\end{equation}
for all sections $X,Y$ of $T^{1,0}M \oplus T^{0,1}M$.
\end{proposition}

In the following, we shall use the following convention for the curvature operator:
\begin{equation}
R(X,Y)Z
=
\nabla_{X} \nabla_{Y} Z
-
\nabla_{Y} \nabla_{X} Z
-
\nabla_{[X,Y]}Z.
\end{equation}
Then for $X, Y, Z$ tangent to $M$,
\begin{align}
\widetilde{R}(X,Y)Z
& =
R(X,Y)Z
+
\sff(X, \nabla_{Y} Z)
-
\sff(Y, \nabla_{X} Z) \notag \\
& \qquad 
+ \widetilde{\nabla}_X (\sff (Y, Z)) 
-
\widetilde{\nabla}_Y (\sff (X, Z))
- \sff([X,Y], Z).
\end{align}
Here, $R$ and $\widetilde{R}$ are the curvature on $(M,\theta)$ and $(N,\eta)$ respectively.
\begin{proposition}[Pseudohermitian Gauß equation]\label{prop:gauss}
Let $\iota \colon (M, \theta) \hookrightarrow (N, \eta)$ be a pseudohermitian submanifold, $\iota^{\ast} \eta = \theta$. Then the Gauß equation holds, i.e.,
	\begin{align}\label{e:gauss}
	\langle \widetilde{R} (X ,\overline{Y}) Z , \Wbar \rangle
	& =
	\langle R (X ,\overline{Y}) Z , \Wbar \rangle
	+
	\left\langle \sff(X, Z) , \sff(\overline{Y}, \Wbar) \right\rangle \notag \\ 
	& \quad -
	|H|^2 \left( \langle \overline{Y} , Z \rangle \langle X , \Wbar \rangle +
	\langle X , \overline{Y} \rangle \langle Z , \Wbar \rangle \right).
	\end{align}
Moreover,
\begin{equation}\label{e:tor}
	\langle \widetilde{\tau} Z , W \rangle
	=
	\langle \tau Z , W \rangle
	+ i \langle \sff(Z,W) , \overline{H} \rangle.
\end{equation}
\end{proposition}
\begin{remark}
For the special case of pseudohermitian \edit{immersions}, this is Eq. (5.3) or Proposition 5.1 in \cite{ebenfelt2004rigidity} as in this case, $H=0$; see also \cite{son2019semi} for similar equations.

It is sometimes helpful to write the Gauß equations using index notations, that is
\begin{equation}
\widetilde{R}_{\alpha\bbar\gamma\bar{\sigma}}
=
R_{\alpha\bbar\gamma\bar{\sigma}}
+
\omega_{\alpha\gamma}^{a}\, \omega_{\bbar\bar{\sigma}}^{\bar{b}}\, g_{a\bar{b}} - |H|^2 \left(g_{\gamma\bbar} g_{\alpha\bar{\sigma}} + g_{\alpha\bbar} g_{\gamma\bar{\sigma}} \right),
\end{equation}
and
\begin{equation}
	\widetilde{A}_{\alpha\beta} = A_{\alpha\beta} + i \omega_{\alpha\beta}^{a} H^{\bar{b}} g_{a\bar{b}}.
\end{equation}
\edit{Here  $\omega_{\alpha\beta}^a$ is the components of the second fundamental form in a local frame, i.e. $\sff(Z_{\alpha},Z_{\beta}) = \omega_{\alpha\beta}^a Z_a$, $H = H^{\bar{b}}Z_{\bar{b}}$, $A_{\alpha\beta} = \langle \tau Z_{\alpha}, Z_{\beta}\rangle$, and so on. These hold} for all isopseudohermitian immersions.
\end{remark}
\begin{proof}[Proof of \cref{prop:gauss}]
The proof is analogous to the case of CR immersions into Kähler manifolds considered in \cite{son2019semi}. Namely, for $X, Z \in T^{1,0}M$ and $\overline{Y}, \Wbar \in T^{0,1}M$, 
\begin{align}\label{e:tem}
\langle \widetilde{R}(X,\overline{Y}) Z, \Wba\rangle
& =
\langle R(X,\overline{Y}) Z, \Wba\rangle 
-
\langle \sff ([X,\overline{Y}], Z), \Wba \rangle \notag \\
& \quad 
+ \langle \widetilde{\nabla}_X (\sff (\overline{Y}, Z)) , \Wba \rangle 
- \langle \widetilde{\nabla}_{\overline{Y}}(\sff (X,Z)) , \Wba \rangle \notag  \\
& =
\langle R(X,\overline{Y}) Z, \Wba\rangle 
-
\langle \sff ([X,\overline{Y}], Z), \Wba \rangle \notag \\
& \quad 
- \langle \sff (\overline{Y}, Z) , \sff (X,\Wba) \rangle 
+ \langle \sff (X,Z) , \sff (\overline{Y},\Wba) \rangle \notag \\
& =
\langle R(X,\overline{Y}) Z, \Wba\rangle 
+ \langle \sff (X,Z) , \sff (\overline{Y},\Wba) \rangle \notag \\
& \quad 
- \langle \overline{Y}, Z \rangle \langle  X,\Wba \rangle |H|^2 
- \langle \sff ([X,\overline{Y}], Z), \Wba \rangle .
\end{align}
On the other hand, by the defining properties of the Tanaka-Webster connection \edit{\cite[Proposition~3.1]{tanaka1975differential}}, we have
\begin{equation}
    [X , \overline{Y}]
    =
    \nabla_X \overline{Y} - \nabla_{\overline{Y}} X - i\langle X , \overline{Y} \rangle T.
\end{equation}
Therefore, using \cref{e:s3},
\begin{equation}
    \langle \sff([X , \overline{Y}] , Z) ,  \Wbar \rangle
    =
    -i\langle X , \overline{Y} \rangle \langle \sff(T, Z) , \Wbar\rangle
    =
    \langle X , \overline{Y} \rangle \langle Z , \Wbar \rangle |H|^2.
\end{equation}
Putting this into \cref{e:tem} we obtain \cref{e:gauss}. Equation \cref{e:tor} follows easily from \cref{e:s3} and we omit the detail.
\end{proof}
\section{The chain rule: Proof of \cref{thm:crintro}}
Let $\iota \colon (M, \theta) \hookrightarrow (N, \eta)$ be a pseudohermitian submanifold, $\iota^{\ast} \eta = \theta$. Recall that the tensor $\nu = \nu_M^N$ introduced in the introduction is the symmetric real tensor on $T^{1,0}M \oplus T^{0,1}M$ defined by
	\begin{equation}
	\nu(X,Y) = 2\left\langle \Sym \sff(X,Y) , \mu \right\rangle - \langle X,Y\rangle |\mu|^2,
	\end{equation}
where $\mu$ is the trace of $\Sym \sff$, i.e., $\mu = \Re H$. The components of $\nu$ in the ``horizontal'' directions are given in local frame by
\begin{equation}
    \nu_{\alpha\beta} = \overline{\nu_{\abar\bbar}} = \omega_{\alpha\beta}^{a} H^{\bar{b}} g_{a\bar{b}},
    \quad
    \nu_{\alpha\bbar} = \nu_{\bbar\alpha}
    =
    \tfrac{1}{2}|H|^2 g_{\alpha\bbar}.
\end{equation}
\edit{Here the Greek indices $\alpha$ and $\beta$ run from $1$ to $n:=\dim_{C\!R}M$, the lowercase indices $a$ and $b$ run from $n+1$ to $\dim_{C\!R} N$, and $g$ is the Levi metric.}

\begin{proposition} Let $\iota \colon (M, \theta) \hookrightarrow (N, \eta)$ be a pseudohermitian submanifold, $\iota^{\ast} \eta = \theta$. Suppose $P$ is the manifold $N$ with the pseudohermitian structure $\widetilde{\eta} = e^u \eta$ and \editb{put} $\widetilde{\theta} = e^{u\circ \iota} \theta$. Then
for any  sections $X, Y$ of $T^{1,0} M \oplus T^{0,1}M$,
\begin{equation}
    \nu_M^P(X,Y)
    =
    	\nu_M^N(X,Y) - 2 \langle \Sym \sff (X,Y) , \xi \rangle_{\theta} + \langle X , Y \rangle_{\theta} |\xi|^2_{\theta}
\end{equation}
where $\xi = \Re (\grad^{1,0}_N u)^{\perp}$.
\end{proposition}
\begin{proof}
Observe that 
\begin{equation}
	\widetilde{\sff}_M^P(X,Y)
	=
	\sff_M^P(X,Y) - \langle X,Y\rangle_{\theta} (\grad^{1,0}_N u)^{\perp},
\end{equation}
and thus, with $\xi = \Re (\grad^{1,0}_N u)^{\perp}$,
\begin{equation}
	\Sym \widetilde{\sff}(X,Y)
	= 
	\Sym \sff(X,Y) - \langle X,Y\rangle_{\theta} \xi.
\end{equation}
Taking the trace, we obtain
\begin{equation}
	e^u \mu_M^P = \mu_M^N - \xi.
\end{equation}
Therefore,
\begin{align}
	\nu_M^P(X,Y)
	& =
	2 \langle \Sym \widetilde{\sff}(X,Y), \mu_M^P \rangle_{\widetilde{\theta}} - \langle X,Y \rangle_{\widetilde{\theta}} |\mu_M^P|^2_{\widetilde{\theta}} \notag \\
	& = 
	2\langle \Sym \sff(X,Y) - \langle X,Y \rangle_{\theta}\, \xi , \mu_M^N  - \xi\rangle_{\theta} -  \langle X,Y \rangle_{\theta} |\mu_M^N - \xi |^2_{\theta} \notag \\
	& = 
	\nu_M^N(X,Y) - 2 \langle \Sym \sff (X,Y) , \xi \rangle_{\theta} + \langle X , Y \rangle_{\theta} |\xi|^2_{\theta}.
\end{align}
The proof is complete.
\end{proof}
Recall that the operator $\mathcal{H}(u)$ on a pseudohermitian manifold $(M,\theta)$ is defined, for horizontal vectors, by
\begin{align}
	\mathcal{H}_{\theta}(u) = \Sym \nabla\nabla u - \partial_b u \otimes \partial_b u - \bar{\partial}_b u \otimes \bar{\partial}_b u + \tfrac{1}{2}|\bar{\partial}_b u|^2 L_{\theta}.
\end{align}
The tensor $\mathcal{H}_{\theta}(u)$ is closely related to the CR Schwarzian tensor \cite{son2018schwarzian}. In the notations of \cite{son2018schwarzian},
\begin{equation}
	B_{\theta}\left(\tfrac{1}{2}u\right) =  \mathcal{H}_{\theta}(u) + \frac{1}{2n} \left( \Delta_b u - n |\bar{\partial}_b u |^2 \right)L_{\theta}.
\end{equation}
is the traceless part of $\mathcal{H}$. Here \edit{$L_{\theta}(X,Y) = \langle X, Y\rangle$ is the Levi metric and} we use the convention that $\Delta_b$ is a nonnegative operator.
We can sometimes write $u$ for $u\circ \iota$ as a function on $M$.
\begin{proposition}[cf. \cite{son2018schwarzian}]\label{prop32}
	Let $u, v \colon M \to \mathbb{R}$ be smooth functions on $M$. Then
	\begin{align}\label{e:1cocycle}
		\mathcal{H}_{\theta}(u + v)
		=
		\mathcal{H}_{\theta} (u) + \mathcal{H}_{\hat{\theta}} (v)
	\end{align}
	where $\hat{\theta} = e^u \theta$.
\end{proposition}
We point out that the equality of the traceless parts of both sides was proved in \cite{son2018schwarzian}.
\begin{proof}
	We need to verify \cref{e:1cocycle} for each pair of vectors of $(1,0)$ and each pair of one $(1,0)$ and one $(0,1)$-vector. First, for any vector $Z,W$ of type $(1,0)$, we have from \cref{e:224}
	\begin{align}
		\mathcal{H}_{\hat{\theta}}(v)(Z, \Wbar)
		& =
		\Sym \widehat{\nabla} v (Z, \Wbar)  + \tfrac12|\bar{\partial}_b v|^2_{\hat{\theta}} \, \langle Z, \Wbar \rangle _{\hat{\theta}} \notag \\
		& =
		\Sym \nabla v (Z, \Wbar)  
		+ \left(\Re \langle \bar{\partial}_b{u} , \partial_b v\rangle 
		+ \tfrac12 |\bar{\partial}_b v|^2_{\theta}\right) \langle Z, \Wbar \rangle _{\theta}
	\end{align}
	Thus,
	\begin{align}
		\mathcal{H}_{\theta}(v)(Z, \Wbar)
		+
		\mathcal{H}_{\hat{\theta}}(v)(Z, \Wbar)
		& =
		\Sym \nabla u(Z, \Wbar) + \tfrac{1}{2}|\bar{\partial}_b u|^2 \langle Z, \Wbar \rangle _{\theta} \notag  \\
		& \quad + \Sym \nabla v (Z, \Wbar)  
		+ \left(\Re \langle \bar{\partial}_b{u} , \partial_b v\rangle 
		+ \tfrac12 |\bar{\partial}_b v|^2_{\theta} \right) \langle Z, \Wbar \rangle _{\theta} \notag \\
		& = \Sym \nabla (u + v)(Z, \Wbar) + \tfrac12 |\partial_b u + \partial_ b v |^2_{\theta} \langle Z, \Wbar \rangle _{\theta} \notag \\
		& = \mathcal{H}_{\theta}( u + v).
	\end{align}
	This verifies \cref{e:1cocycle} for any pair of vectors of mixed type. The identity for vectors of pure type is exactly the same as \cite{son2018schwarzian}. We omit the details.
\end{proof}
\begin{proposition} \editb{If} $(M,\theta) \subset  (N,\eta)$ is a pseudohermitian submanifold, $X, Y \in \Gamma (T^{1,0} \oplus T^{0,1} M)$, and $u \in \mathcal{C}^{2}(N)$, then
\begin{align}
	\mathcal{H}_{\theta} (u|_M)(X,Y) - \mathcal{H}_{\eta}(u)(X,Y) 
	=
	(du)(\Sym \sff (X,Y)) - \tfrac{1}{2} \left|(\grad_{N,M}^{1,0} u )^{\perp}\right|^2 \langle X, Y \rangle _{\theta}.
\end{align}
\end{proposition}
\begin{proof}
On $M$, we have 
\begin{equation}
	\nabla^2 (u|_M)(X,Y)
	=
	X(Y(u)) - d(u|_M)(\nabla_XY),
\end{equation}
and similarly for $\widetilde{\nabla}^{2} u(X,Y)$ on $N$. Thus,
\begin{equation}
	\nabla^{2} u(X,Y) - \widetilde{\nabla}^{2} u(X,Y) 
	=
	(du)(\sff(X,Y)).
\end{equation}
Consequently,
\begin{align}
	\mathcal{H}_{\theta}(u|_M)(X,Y) - \mathcal{H}_{\eta}(u)(X,Y) 
	& = \tfrac{1}{2} (du)(\sff(X,Y) + \sff (Y,X)) \notag \\
	& \quad  +\tfrac{1}{2}\left(|\partial_b (u|_M)|^2 - |\partial_ b u|^2 \right) \langle X, Y \rangle_{\theta} \notag \\
	& =
	(du)(\Sym \sff(X,Y)) - \tfrac{1}{2} |(\grad_{N,M}^{1,0} u )^{\perp}|^2 \langle X, Y \rangle _{\theta}. \qedhere 
\end{align}
\end{proof}
\begin{proposition}\label{prop:e}
\editb{For any}  tower of pseudohermitian submanifolds $M\subset N \subset (P,\eta)$ \editb{ it holds that }
\begin{align}\label{e:mean}
	\mu_M^{P}
	& =
	\mu_M^N + \mu_N^P,  \\ \label{e:nu}
	\nu_M^P - \nu_M^N
	& =
	(\iota_M^N)^{\ast} \nu_N^P.
\end{align}
\end{proposition}
\begin{proof}
	For $Z\in T^{1,0}M$ and $\Wbar \in T^{0,1}M$, we have that
	\begin{equation}
		\langle \Zbar , W \rangle H_M^P
		=
		\sff_M^P(\Zbar, W) 
		= 
		\sff_N^P (\Zbar,W) + \sff_M^N(\Zbar,W)
		=
		\langle Z , \Wbar \rangle (H_N^P + H_M^N)
	\end{equation}
	Taking $Z = W \ne 0$, we immediately obtain $H_M^P = H_M^N + H_N^P$ and hence \cref{e:mean} follows. Plugging this into the definition of $\nu$, we obtain \cref{e:nu}. The proof is complete.
\end{proof}
\begin{proof}[Proof of \cref{thm:crintro}] The idea of the proof is essentially the same as in \cite{stowe2015ahlfors} and based on the calculations above. The formula in the  CR case turns out to be simpler than its conformal counterpart because of \cref{prop:e} above. Indeed,  assume that $F \colon (M,\theta) \to (N,\eta)$ and $G\colon (N,\eta) \to (P,\zeta)$. The chain rule for $\ahlfors$ is equivalent to the analogous assertions for $\ahlfors'$ where
\begin{equation}
    \mathcal{A}' \colon F \mapsto \mathcal{H}_{\theta} (u) + F^{\ast} \left(\nu_{F(M)}^N\right).
\end{equation}
In the equidimensional case, this and \cref{e:chainrule} reduce to an analogous statement for \edit{the} CR Schwarzian that was proved in  \cite{son2018schwarzian}.

Next, we \editb{assume that} $P$ and $N$ have the same dimension. In this case, we can suppose that $P$ is the manifold $N$ with a 
pseudohermitian structure $\zeta = e^{v} \eta$ and $G$ is the identity map. Since $e^u\theta = F^{\ast} \eta$, by \cref{prop32},
\begin{align*}
    \mathcal{H}_{\theta} (u + v \circ F)
    & =
    \mathcal{H}_{\theta}(u) + \mathcal{H}_{e^{u} \theta} (v \circ F) \\
    & =
    \mathcal{H}_{\theta}(u) + F^{\ast}\mathcal{H}_{(F(M),\iota^{\ast} \eta)} (v).
\end{align*}
Therefore,
\begin{align}
    \ahlfors'(G \circ F) - \ahlfors'(F) - F^{\ast} \ahlfors'(G)
    & = \mathcal{H}_{\theta}(u + v\circ F) + F^{\ast} \left(\nu_{F(M)}^{(N, e^v \eta)}\right) \\
    & \quad - \left(\mathcal{H}_{\theta}(u) + F^{\ast} \left(\nu_{F(M)}^{(N,\eta)}\right)\right) - F^{\ast} \left(\mathcal{H}_{\eta}(v)\right)\notag \\
    & = F^{\ast} \left(\mathcal{H}_{(F(M),\iota^{\ast} \eta)} (v) - \mathcal{H}_{\eta} (v) + \nu_{F(M)}^{(N, e^v \eta)} - \nu_{F(M)}^{(N,\eta)}\right) \notag \\
    & = 0.
\end{align}
To conclude the proof, we consider the general case when $G$ is a CR immersion, $G^{\ast} \zeta = e^{v}\eta$. We have
\begin{align}
	\mathcal{A}'(G\circ F) - \mathcal{A}'(F) - F^{\ast} (\mathcal{A}'(G))
	& =
	\mathcal{H}_{\theta}(u + v \circ F) + (G\circ F)^{\ast}(\nu_{F(G(M))}^P) \notag \\
	& \quad - \left(\mathcal{H}_{\theta}(u) + F^{\ast}(\nu_{F(M)}^N)\right) 
	- F^{\ast}\left(\mathcal{H}_{\eta}(v) + G^{\ast}(\nu_{G(N)}^P\right) \notag \\
	& = (G\circ F)^{\ast}\left(\nu_{G(F(M))}^P - \nu_{G(N)}^P\right) - F^{\ast} (\nu_{F(M)}^N) \notag  \\
	& = 0.
\end{align}
At the last step, we have used \cref{prop:e}. The proof is complete.
\end{proof}

\section{CR maps with vanishing Ahlfors derivative}
It is well-known that the Schwarzian derivative of a conformal diffeomorphism measures the change of the traceless component of the Ricci tensor, see \cite{osgood1992schwarzian}. The same is true for the CR analogue of the Schwarzian \cite{son2018schwarzian}. The case of immersions is a bit different due to the presence of the second fundamental form. We point out that \edit{in the CR case a certain} part of the Ahlfors \edit{derivative} actually measures how the pseudohermitian torsion changes when going from the original structure to the pull-back.

We denote $\ahlfors_{\alpha\beta}(F) = \ahlfors(F)(Z_{\alpha}, Z_{\beta})$ the ``holomorphic''  components and $\ahlfors_{\alpha\bbar}(F) = \ahlfors(F)(Z_{\alpha}, Z_{\bbar})$ the ``mixed type'' components of the Ahlfors \edit{derivative}. \editc{As usual, $h_{\alpha\bbar}$ denotes the Levi metric on $(M,\theta)$, i.e. $h_{\alpha\bbar} = \langle Z_{\alpha} , Z_{\bbar}\rangle = -i d\theta(Z_{\alpha}, \overline{Z_{\beta}})$, and similarly, $g_{A\bar{B}}$ denotes the one for $(N,\eta)$.}
\begin{proposition} Let $F \colon (M,\theta) \to (N,\eta)$ be a CR immersion. Let $A_{\alpha\beta}$ be the \edit{pseudohermitian} torsion of $(M,\theta)$ and $\widetilde{A}_{AB}$ the pseudohermitian torsion of $(N,\eta)$. Also let $R_{\alpha\bbar}$ be the Ricci tensor of the Tanaka-Webster connection associated to $\theta$. In an adapted coframe \edit{the following hold\editb{:}}
\begin{enumerate}[\rm (i)]
	\item The ``holomorphic'' components of the Ahlfors \edit{are}
	\begin{equation}\label{e:holomorphicpart}
	\ahlfors_{\alpha\beta}(F)
	=
	-i(\widetilde{A}_{\alpha\beta} - A_{\alpha\beta}).
	\end{equation}
	\item The tracefree components of \editb{mixed type} \edit{are}
	\begin{equation}
		\tf \ahlfors_{\alpha\bbar} (F) 
		=
		\tf R_{\alpha\bbar} - \tf \hat{R}_{\alpha\bbar},
	\end{equation}
	where $\hat{R}_{\alpha\bbar}$ is the Ricci tensor associated to $\hat{\theta} = F^{\ast}\eta$.
	\item If $(N,\eta)$ is the sphere with its standard pseudohermitian structure, then
	\begin{equation}\label{e:traceahlfors}
		\trace_{\theta}\ahlfors(F):= h^{\alpha\bbar} \ahlfors_{\alpha\bbar}(F)
		=
		\tfrac{1}{2(n+1)} e^u |\sff^{2,0}|^2,
	\end{equation}
	where $|\sff^{2,0}|^2$ is the \edit{squared norm} of the CR second fundamental form. % in the metric $\hat{h} \otimes g$.
\end{enumerate}
\end{proposition}
\begin{remark}
	This proposition exhibits some new features of the CR analogue of the Ahlfors derivative compared to its conformal counterpart. Observe that (1) and (2) generalize similar formulas for the CR Schwarzian in \cite{son2018schwarzian} while (3) says that the Ahlfors is traceless if and only if $F$ has vanishing CR second fundamental form.
	\edit{Moreover, part (iii) also holds when $(N,\eta)$ has the Tanaka-Webster curvature tensor of the form
		\[
		R_{A\bar{B}C\bar{D}}
		=
		c\left(g_{A\bar{B}} g_{C\bar{D}} + g_{A\bar{D}} g_{C\bar{B}}\right), \ c\in \mathbb{R}.
		\]
	In particular, it holds when $(N,\theta)$ is the Heisenberg hypersurface with \editb{a} pseudohermitian structure of vanishing curvature and torsion.}
\end{remark}
\begin{proof} We identify $M$ with $F(M)$ and consider $\iota \colon (M,\hat{\theta}) \to (N,\eta)$ as \editb{the inclusion of} a submanifold and thus $\hat{\theta}: = \iota^{\ast}\eta =  e^u\theta$ for some smooth function~$u$. Let $\hat{A}_{\alpha\beta}$ be the pseudohermitian torsion of $\hat{\theta}$ on $M$. By the Gauß equation, we have
\begin{equation}
\widetilde{A}_{\alpha\beta} = \hat{A}_{\alpha\beta} + i\, \omega_{\alpha\beta}^a H^{\bar{b}} g_{a\bar{b}}.
\end{equation}
On the other hand, $A_{\alpha\beta} = \hat{A}_{\alpha\beta} - i(u_{\alpha,\beta} - u_{\alpha} u_{\beta})$, by \cite{lee1986fefferman}. \edit{Here the indices preceded by a comma indicate covariant derivatives.} Therefore,
\begin{align}
\ahlfors_{\alpha\beta}(F)
&=
u_{\alpha,\beta} - u_{\alpha} u_{\beta} + \omega_{\alpha\beta}^a \xi^{\bar{b}} g_{a\bar{b}} \notag \\
&=
u_{\alpha,\beta} - u_{\alpha} u_{\beta} -i(\widetilde{A}_{\alpha\beta} - \hat{A}_{\alpha\beta}) \notag \\
&=
-i(\widetilde{A}_{\alpha\beta} - A_{\alpha\beta}).
\end{align}
Thus, (1) is proved.

From the definition, the \editb{mixed type} components are given by
\begin{equation}\label{e:tracepart}
\ahlfors_{\alpha\bbar} (F)
=
\tfrac{1}{2}\left(u_{\alpha,\bbar} + u_{\bbar,\alpha}\right) +  \tfrac{1}{2} \left(|\bar{\partial}_b u|^2 + e^u |H_{F(M)}|^2 \circ F + e^u J_{\eta} \circ F  - J_{\theta}\right) h_{\alpha\bbar}.
\end{equation}
Therefore, using \cite{lee1986fefferman}, we obtain
\begin{equation}
	\tf \ahlfors_{\alpha\bbar} (F) 
	=
	\tfrac{1}{2}\left(u_{\alpha,\bbar} + u_{\bbar,\alpha}\right) + \frac{1}{2n} \Delta_b u\, h_{\alpha\bbar}
	=
	\tf R_{\alpha\bbar} - \tf \hat{R}_{\alpha\bbar},
\end{equation}
where $\hat{R}_{\alpha\bbar}$ is the Ricci curvature of $\hat{\theta} = e^u \theta = F^{\ast}\eta$.

Suppose that $(N,\eta)$ is a sphere with the standard pseudohermitian structure, then \edit{the Gauß equation} implies that
\begin{equation}
	J_{\eta} \circ F
	=
	J_{\hat{\theta}} 
	- |H_{F(M)}|^2 \circ F + \frac{1}{n(n+1)}|\sff_{F(M)}^{0,2}|^2 \circ F.
\end{equation}
On the other hand,
\begin{equation}
	e^u J_{\hat{\theta}}
	=
	J_{\theta} +\tfrac{1}{n} \Delta_b u - |\bar{\partial}_b u|^2.
\end{equation}
Putting them together, we obtain that
\begin{equation}
	e^uJ_{\eta} \circ F - J_{\theta} + |\bar{\partial}_b u|^2 + e^u |H_{F(M)}|^2 \circ F -\frac{1}{n} \Delta_b u = \frac{1}{n(n+1)} e^u |\sff^{2,0}_{F(M)}|^2 \circ F.
\end{equation}
\editb{It thus holds that}
\begin{equation}\label{e:traceahlfors}
	\trace_{\theta}{\ahlfors(F)}
	=
	h^{\alpha\bbar} \ahlfors_{\alpha\bbar}(F)
	=
	\frac{1}{2(n+1)} e^u |\sff^{2,0}|^2, 
\end{equation}
\editb{and the} proof is complete.
\end{proof}
\begin{definition}
	Let $\iota \colon (M,\theta) \hookrightarrow (N,\eta)$ be a pseudohermitian submanifold, $\iota^{\ast} \eta = \theta$. We say that \editb{$\iota$ (or $M$ if the embedding is understood)} is umbilic at $p$ if $\sff_M^N(p)$ is a multiple of the Levi metric, i.e., $\sff(Z,W) = 0$ for all $(1,0)$ vectors $Z$ and $W$ at~$p$.
\end{definition}
We remark that \edit{the pseudohermitian total umbilicity does not depend on the choice of pseudohermitian structures}. In fact, if $\iota$ is pseudohermitian umbilic, then it is CR totally geodesic for any admissible pair of contact forms. This \editb{holds} because the CR second fundamental form (as defined in \cite{ebenfelt2004rigidity}) associated to any admissible pair of pseudohermitian structures coincides with the ``holomorphic'' components of the pseudohermitian second fundamental form.
\begin{corollary}
	Let $(M,\theta)$ be a pseudohermitian manifold and \editb{assume that} $F \colon (M,\theta) \to (\mathbb{S}^{2N+1},\Theta)$ a CR immersion into a sphere. Then $\trace\ahlfors(F)\geq 0$, \editb{and} equality holds at \editb{a point} $p\in M$ iff $(M,\theta)$ is umbilic at~$p$. If $\trace\ahlfors(F) = 0$ identically, then $M$ is CR spherical and $F$ is spherically equivalent to the linear embedding. 
\end{corollary}
\begin{proof}
	That $\trace \ahlfors(F) \geq 0$ follows directly from \cref{e:traceahlfors}. The equality occurs precisely when the CR second fundamental form vanishes. This implies that $M$ must be CR spherical by \cite{ebenfelt2004rigidity} for $n\geq 2$ and \cite{son2019semi} for the case $n=1$. Moreover, $F$ is equivalent to the linear embedding by \cite{ji2010flatness} for the case $\dim M \geq 5$ and \cite{son2019semi} for the three-dimensional case.
\end{proof}
\section{Maps from three-dimensional CR manifolds into spheres}

Suppose that $(N, \eta) = (\mathbb{S}, \Theta)$ where $\Theta$ is the standard pseudohermitian structure \edit{on the sphere} \editb{$\mathbb{S} := \mathbb{S}^{2N+1}$} . If $F = \phi \circ G$ for some CR automorphism $\phi$ of the target sphere, then $\ahlfors(F) = \ahlfors(G)$ by \cref{thm:crintro} and, in particular, $\ahlfors^{\theta}_{\alpha\bbar}(F) = \ahlfors^{\theta}_{\alpha\bbar}(G)$. This last equality implies that $u_{\alpha,\bbar} = v h_{\alpha\bbar}$ for some function $v$. When $n\geq 2$, \edit{using Lee's characterization of CR-pluriharmonicity \cite{lee1988pseudo}, we \editb{can infer} from the last condition that $u$ must be CR-pluriharmonic}. This argument does not work in the case $n=1$. We shall deduce the CR-pluriharmonicity from a simpler argument which also works in the case $n=1$ as follows.

\begin{proposition}\label{prop:crph}
    Let $M$ be a strictly pseudoconvex CR manifold and  assume that \editb{$F$ and $G$ are CR immersions from $M$ into the sphere $\mathbb{S}$}. Let $\theta = F^{\ast} \Theta$ and $\theta' = G^{\ast} \Theta$, where $\Theta$ is the standard pseudohermitian structure on $\mathbb{S}$. Put $\theta' = e^{u}\theta$. If $F =  \phi \circ G$, for some CR automorphism $\phi$ of $\mathbb{S}$, then $u$ is CR-pluriharmonic.
\end{proposition}
\begin{proof} 
    If $G = \phi\circ F$ for some CR automorphism $\phi$ of $\mathbb{S}$, then this automorphism extends to \edit{a} biholomorphic automorphism of the ball $\mathbb{B}$ which we will also denote by $\phi$. Let $\phi(0) = a$ and let $\rho = \|z\|^2 - 1$ be the defining function for $\mathbb{S}$. Then by \cite{rudin1980function}, Theorem 2.2.2,
     \begin{equation}
	    (\rho \circ \phi)(z)
	    =
	    |\phi(z)|^2 - 1
	    =
	    \left(\frac{1 - |a|^2}{|1 - z\cdot \bar{a}|^2}\right) \cdot \rho(z).
     \end{equation}
    Observe that $\Theta = -\iota^{\ast}( i\partial \rho)$ and hence $\phi^{\ast}\Theta = -\iota^{\ast}(i \partial (\rho \circ \phi)) = e^{\varphi}\Theta$, where 
    \begin{equation} 
	    \varphi (z) = \log (1 - |a|^2) -  \log \left|1 -  z\cdot \bar{a}\,\right|^2, \quad z \in \mathbb{S}.
    \end{equation} 
	In particular, $\varphi$ is CR-pluriharmonic and so is $u = \varphi \circ F$.
\end{proof}
In view of this proposition, we can construct another invariant which plays the role of the traceless part of $\ahlfors(F)$ as follows. By \cite[Proposition~3.4]{lee1988pseudo}, $u$ is CR-pluriharmonic if and only if
\begin{equation}
    P u : = \left(u_{\bar{1},}{}^{\bar{1}}{}_{1} + iA_{11} u^1\right)\, \theta^1 =0.
\end{equation}
Here $A_{11}$ is the component of the pseudohermitian torsion. This suggests the following notion: For each CR immersion $F$ from the $3$-sphere into a sphere, define the $(1,0)$-form $\ahlfors_1(F)$ by
\begin{equation}
    \mathcal{A}_{1}(F) = P(u),
\end{equation}
where as before $u$ is determined by the relation $ F^{\ast} \Theta = e^u\theta$.

We formulate the discussion above as follows.
\begin{proposition}\label{prop:3dinvariant}
    Let $(M,\theta)$ be a three-dimensional strictly pseudoconvex pseudohermitian manifold \editb{and} assume that $F$ and $G$ are CR immersions of $M$ into \editb{the} sphere \editb{$\mathbb{S}$}. Then $\mathcal{A}_{1}$ \edit{is invariant with respect to } \editb{ spherical equivalence} in the following sense:
    \begin{enumerate}[\rm(1)]
    	\item If $\phi$ is a CR automorphism of the target, then
    	\begin{equation}
    	\mathcal{A}_{1}(\phi \circ F) = \mathcal{A}_{1}(F).
    	\end{equation}
    	\item If $M$ is \editb{the 3-sphere} and $ \gamma$ is a CR automorphism of \editb{$M$}, then
    	\begin{equation}\label{e:ahlforsprime}
    	\ahlfors_1(F\circ \gamma) 
    	 = \gamma\, ^{\ast}\left(e^{\psi}\mathcal{A}_{1}(F) \right)
    	\end{equation}
    	where $\psi$ is determined by the relation $\gamma^{\ast} \theta = e^{\psi}\theta$.
    \end{enumerate} 
\end{proposition}
\begin{proof}
	We write $F^{\ast}\Theta = e^u\theta$ and $G^{\ast}\Theta = e^v\Theta$. Suppose that $G = \phi \circ F$, then \cref{prop:crph} implies that $u - v$ is CR pluriharmonic and hence
	\begin{equation} 
		\ahlfors_1(\phi \circ F) = P(v) =  P(u) = \ahlfors_1(F),
	\end{equation} 
	as desired.
	
	Suppose that $M$ is a 3-sphere, $\gamma$ a CR automorphism of $M$, and $\theta$ is the standard pseudohermitian structure on $M$. Suppose that $G = F\circ \gamma$, then
	\begin{equation} 
		v = u\circ \gamma + \psi.
	\end{equation}
	Under the change of contact form $\widetilde{\theta} = e^{\psi}\theta$, the operator $P$ changes as follows:
	We write $P^{\theta}$ for the operator $P$ associated to $\theta$ and similarly for $\widetilde{\theta}$. Then 
	\begin{equation} 
		e^{\psi}P^{\widetilde{\theta}}(u)
		=
		P^{\theta}(u) + 2i \left\langle P^{\theta}(u) , \bar{\partial}_b \psi \right\rangle \, \theta.
	\end{equation} 
	Since $P(\psi) = 0$, we have that
	\begin{equation} 
		P^{\theta}(v) 
		=
		P^{\theta}(u \circ \gamma)
		=
		\gamma^{\ast}\left(P^{(\gamma^{-1})^{\ast}\theta} (u)\right) 
		=
		e^{\psi} \gamma^{\ast} \left(P^{\theta}(u)\right) \mod \theta.
	\end{equation} 
	This completes the proof.
\end{proof}
\section{Explicit calculations in local coordinates}
\subsection{The Ahlfors tensor $\ahlfors$ in general dimensions}
The purpose of this section is to give an explicit formula for the hermitian part of the Ahlfors tensor. The formula will be explicit in terms of the defining functions of the source and target and will be simplified when the target is the sphere. Let $M \subset \mathbb{C}^{n+1}$ be a strictly pseudoconvex real hypersurface defined by $\rho = 0$. It is well-known that there exists \editb{a unique} $(1,0)$ vector field $\xi$ satisfying the following two conditions (see, e. g. \cite{graham1988smooth,li--son}):
	\begin{equation}
	\xi \, \rfloor \, i\partial\bar{\partial} \rho = ir \bar{\partial} \rho, \quad 
	\partial\rho(\xi) = 1.
	\end{equation}
The function $r \editb{=r[\rho]}: = \rho_{j\kbar}\xi^j \xi^{\kbar} $ is called the transverse curvature of $M$ and $\rho$ \cite{graham1988smooth}. When $\rho$ is strictly plurisubharmonic, then $r^{-1} = |\partial \rho|^2$ in the Kähler metric with potential $\rho$.

\begin{proposition}\label{prop:mcnt}
	Let $M$ be a strictly pseudoconvex real hypersurface in $\mathbb{C}^{n+1}$ and $F\colon M \to N \subset\mathbb{C}^{d+1}$ a CR immersion. Suppose that $F$ extends holomorphically to a neighborhood of $M$. Let $\widetilde{\rho}$ be a strictly plurisubharmonic defining function for $N$ \editb{and let}
 $\eta = i\bar{\partial}\rho$, $\rho = \widetilde{\rho} \circ F$, $\theta = i\bar{\partial}\rho$, \editb{so that } $F^{\ast} \eta = \theta$. Then  
	\begin{equation}
	|H_{F(M)}|^2 \circ F = r[\rho] - r[\widetilde{\rho}] \circ F.
	\end{equation}
	In particular, $r[\widetilde{\rho}] \circ F \leq r[\rho]$ on $M$. The equality holds if and only if $(\theta,\eta)$ is an admissible pair for the CR immersion $F$.
\end{proposition}
\begin{proof} \editb{As already noted,}  $F^{\ast} \eta = \theta$ and \editb{also note} that the Reeb vector field is $T = i(\xi - \bar{\xi})$. We shall compute $F_{\ast}{\xi}$ as follows.  \edit{In local coordinates} \editb{$(z_1, \dots, z_{n+1})$}
%$\{z_j, j=1,2,\dots, n+1\}$, 
we write $\xi =  \xi^j \partial_j$ (\editb{using} summation convention). By direct calculations (see, e.g., \cite{lee--melrose,li--son}) \editb{we obtain}
	\begin{equation}
	\xi^j = r \rho^j = r \rho^{j\bar{k}} \rho_{\bar{k}},
	\end{equation}
	\edit{where} $\rho_j = \partial\rho /\partial z_j$, $\rho_{j\bar{k}} = \partial^2 \rho/\partial \bar{z}_k \partial z_j$, and $\rho^{j\bar{k}}$ is the inverse \edit{transpose} matrix of $\rho_{j\bar{k}}$. Thus, for $p\in M$ and $q = F(p) \in M\subset N$, and in local coordinates $z'^A$,
	\begin{equation}
	F_{\ast} (\xi_p)
	=
	F^A_j (p)\,\xi^j_p\, \partial_A.
	\end{equation}
	On the other hand, the Reeb vector field on $N$ is given by $\widetilde{T}  = i(\xi' - \overline{\xi'})$, with $\xi' = \xi'^A\partial_A$. Thus, by \cref{prop:ttprime},
	\begin{equation}
	H_{F(M)}|_{F(p)}
	=
	\left(F^A_j (p) \xi^j_p - \xi'^A\right) \partial_A.
	\end{equation}
	By direct calculations,
	\begin{equation}
	\rho_j = \widetilde{\rho}_A F^A_j, 
	\quad 
	\rho_{j\bar{k}} 
	=
	F^A_j F^{\bar{B}} _{\bar{k}} \widetilde{\rho}_{A\bar{B}}.
	\end{equation}
	Therefore,
	\begin{equation}
	F^A_j \xi'^{\bar{B}}\widetilde{\rho}_{A\bar{B}}
	=
	\widetilde{r}F^A_j\widetilde{\rho}_{A}
	=
	\widetilde{r}\rho_j.
	\end{equation}
	Consequently,
	\begin{align}
	|H_{F(M)}|^2 \circ F
	& =
	\left(\xi'^A - F^A_j \xi^j \right) \left(\xi'^{\bar{B}} - F^{\bar{B}}_{\bar{k}} \xi^{\bar{k}}\right) \widetilde{\rho}_{A\bar{B}}\notag \\
	& =
	\xi'^A\xi'^{\bar{B}} \widetilde{\rho}_{A\bar{B}}
	-2 \Re \left(F^A_j \xi^j \xi'^{\bar{B}}\widetilde{\rho}_{A\bar{B}} \right)
	+ 
	F^A_jF^{\bar{B}}_{\bar{k}}\xi^j\xi^{\bar{k}} \widetilde{\rho}_{A\bar{B}}\notag \\
	& =
	\widetilde{r} -2\Re \left(\widetilde{r} \xi^j \rho_j \right) + \xi^j\xi^{\bar{k}} \rho_{j\bar{k}} \notag\\
	& =
	\widetilde{r} -2 \widetilde{r} + r \notag \\
	& 
	= r - \widetilde{r}.
	\end{align}
	The proof is complete.
\end{proof}
\editb{We also need the following computational result.}
\begin{proposition}[\cite{son2019semi}]\label{prop:fun}
	Let $M$ be a strictly pseudoconvex real hypersurface defined by $\rho = 0$, with $d\rho\ne 0$ along $M$. Let $\sigma$ be a smooth function in a neighborhood of $M$ and $\hat{\rho} = e^{\sigma} \rho$. Then
	\begin{equation}
	e^{\sigma} \hat{r} 
	=
	r + 2\Re(\xi) \sigma - |\bar{\partial}_b \sigma|^2.
	\end{equation}
	Here, $\bar{\partial}_b$ is the Cauchy-Riemann operator associated to $\theta: = -i\partial \rho$ and $\xi$ is the transverse vector field of $\rho$.
\end{proposition}
Suppose that $\rho$ and $\widetilde{\rho}$ are defining functions for $(M,\theta)$ and $(N,\eta)$ such that 
$\theta = i\bar{\partial}\rho$ and $\eta = i\bar{\partial} \widetilde{\rho}$. Let $F\colon M \to N$ be a CR immersion, extended as a holomorphic immersion in a neighborhood of a point $p\in M$, and let $Q$ be such that 
\begin{equation}
	\widetilde{\rho} \circ F = Q \rho.
\end{equation}
By a transversality argument, $Q\ne 0$ on $M$. Thus, we may assume that $Q> 0$ on $M$. Then $F^{\ast} \eta = e^u\theta$ with $u = \log Q|_M$.

If $r[\rho]$ and $J[\rho]$ are the transverse curvature of $\rho$ and the Levi-Fefferman determinant (see \cite{li--son}), i.e. 
\begin{equation}
	J(\rho)
	=
	-\det 
	\begin{bmatrix}
		\rho & \rho_{\bar{k}} \\
		\rho_{j} & \rho_{j\kbar}
	\end{bmatrix},
	\quad
	r(\rho) = \frac{\det[\rho_{j\kbar}]}{J(\rho)},
\end{equation}
then by the Li-Luk formula for the Webster scalar curvature \cite{li--luk} (see also \cite[Proposition~4.1]{li--son}),
\begin{equation}
	J_{\theta}
	=
	r(\rho) + P_{\rho} \log J(\rho),
\end{equation}
where 
\begin{equation} 
	P_{\rho} : = \frac{1}{2n(n+1)}(\xi^j\xi^{\kbar} - \psi^{j\kbar})\partial_j \partial_{\kbar}.
\end{equation}
Then by using \cref{prop:mcnt,prop:fun} and the formulas above, we obtain that, in terms of the local frame $Z_\alpha:= \partial_{\alpha} - (\rho_{\alpha}/\rho_w)\partial_w$, the \editb{mixed type} \edit{components} of $\ahlfors$ are given by
\begin{align}\label{e614}
	\ahlfors_{\alpha\bbar}(F)
	=
	\frac{1}{2}\left(u_{\alpha,\bbar} + u_{\bbar,\alpha}\right) + \frac12 \bigl(N_{\rho}  u\bigr) h_{\alpha\bbar} + \frac{1}{2} \bigl(P_{\widetilde{\rho}} \log J(\widetilde{\rho}) - P_{\rho} \log J(\rho) \bigr) h_{\alpha\bbar},
\end{align}
where $u = \log Q$. 

Following \cite{li--luk}, we define the second order operator
\begin{equation}
D^{\rho}_{\alpha\bbar} 
=
\partial_{\bbar}\partial_{\alpha}
-(\rho_{\alpha}/\rho_{w}) \partial_w \partial_{\bbar} - (\rho_{\bbar}/\rho_{\wbar}) \partial_{\wbar}\partial_{\alpha} + (\rho_{\alpha}\rho_{\bbar}/|\rho_w|^2) \partial_{w}\partial_{\wbar},
\end{equation}
which satisfies 
\begin{equation} 
	D^{\rho}_{\alpha\bbar} \varphi = \varphi_{Z\Zbar}(Z_{\alpha}, Z_{\bbar}).
\end{equation} 
\editb{We can now give a completely explicit formula for the \editb{mixed type} components of the 
 Ahlfors tensor.} 
\begin{proposition}\label{prop:explicitahlfors} With the notations \editb{introduced} above, it holds that
\begin{equation}
	\ahlfors_{\alpha\bbar}(F)
	=
	D^{\rho}_{\alpha\bbar} \log Q - \frac{1}{2}\bigl(P_{\widetilde{\rho}} \log J(\widetilde{\rho}) - P_{\rho} \log J(\rho) \bigr) h_{\alpha\bbar}.
\end{equation}
\end{proposition}
\begin{proof}
	This follows from \cref{e614} and a well-known formula for the Christoffel symbols of the Tanaka-Webster connection \cite{li--luk}.
\end{proof}
We point out that this formula involves both tangential and normal derivatives of the quotient $Q$ and the Fefferman \edit{determinants} on the source and the target. The interesting case is when $\rho$ and $\widetilde{\rho}$ are approximate Fefferman defining functions for the source and target of order $3$, i.e., when $J(\rho) = 1 + o(\rho^{2})$ and $J(\widetilde{\rho}) = 1 + o(\widetilde{\rho}^{2})$, because then the formula simplifies to \begin{equation} \ahlfors_{\alpha\bbar}(F) = D^{\rho}_{\alpha\bbar} \log Q.\end{equation} 
In particular, we have a particularly simple formula when both source and target are spheres. 
\begin{corollary}
	Let $(M,\theta)$ be a strictly pseudoconvex pseudohermitian manifold, $F$ and $G$ nonconstant holomorphic maps sending $M$ into $\mathbb{S}^{2N+1}$, and
	\begin{equation}
		\|F\|^2 - 1 = Q_F\cdot \rho, 
		\quad
		\|G\|^2 - 1 = Q_G \cdot \rho. 
	\end{equation}
	where $\rho$ is a defining function for $M$, $\theta: = i\bar{\partial}\rho$. If $F = \gamma \circ G$ for some $\gamma \in \aut(\mathbb{S}^{2N+1})$, then
	\begin{equation}\label{e:a}
		\iota^{\ast} \partial\bar{\partial} \log (Q_F/Q_G)
		= 0.
	\end{equation}
	Here $\iota \colon M \to \mathbb{C}^{n+1}$ is the inclusion.
\end{corollary}
\begin{proof}
	As explained above, in the frame $Z_{\alpha} = \partial_{\alpha} - (\rho_{\alpha}/\rho_w)\partial_w$, we have that
	\begin{equation} 
		\ahlfors_{\alpha\bbar}(F) = D^{\rho}_{\alpha\bbar} \log Q_F + \frac{1}{2}P_{\rho} \log J(\rho) h_{\alpha\bbar},
	\end{equation} 
	and similar for $\ahlfors_{\alpha\bbar}(G)$. Thus, if there exists such a $\gamma$, then by \cref{thm:crintro}, we have that
	\begin{equation} 
		D^{\rho}_{\alpha\bbar} \log (Q_F/Q_G) = 0.
	\end{equation} 
	The last equality is clearly equivalent to \cref{e:a}, \editb{ and the proof} is complete.
\end{proof}
Note that if $F$ and $G$ are nonconstant then $Q_F$ and $Q_G$ are nonvanishing on $M$ by a Hopf Lemma (a transversality result). The quotient $Q_F$ defined as above has been used extensively in the study of proper holomorphic maps between balls or sphere maps; see, e.g., \cite{d1993several,d2017symmetries} and the references therein.

The explicit formula can be used to deduce a necessary condition, which is simple to check, for a map to be equivalent to a monomial map as follows.

\begin{corollary}\label{cor:monomial}
	Let $F$ be a monomial map between spheres in the standard coordinates, then $\trace \ahlfors(F)$ only depends on the moduli \editb{ of the } $|z_j|$, $j=1,2,\dots, n + 1$. Moreover, for each~$j$, $\ahlfors_{\alpha\bbar}(F)$ is tracefree on the set $|z_j| = 1$. Consequently, if $F$ is equivalent to a monomial map, then the umbilical locus of $F(\mathbb{S}^{2n+1})$, if not empty, is invariant under a torus action.
\end{corollary}
We mention a similar result characterizing the monomial maps $F$ in terms of the invariant group $\Gamma_F$ in a recent paper by D'Angelo and Xiao \cite{d2017symmetries}. Precisely, they proved that $F$ is \edit{spherically equivalent to a monomial map} iff the group $\Gamma_F$ contains an $n$-torus.
\begin{proof}
	Suppose that $F$ is a monomial map, then $\|F\|^2 -1$ is a multivariate polynomial in $|z_j|^2$, $j=1,2,\dots , n+1$: There is a polynomial $f(t_1,\dots , t_{n+1})$ such that 
	\begin{equation} 
		\|F\|^2 - 1 = f(|z_1|^2,\dots , |z_{n+1}|^2).
	\end{equation} 
	Since $F$ sends sphere into sphere, $f(t)$ is divisible by $t_1 + t_2 +\dots + t_{n+1} - 1$ in the polynomial ring $\mathbb{R}[t]$. Thus, the quotient $g(t_1,\dots, t_{n+1})$ is a polynomial in $t_j$ with real coefficients. As $\log Q_F(z) = \log g(|z_1|^2, \dots , |z_{n+1}|^2)$, we can write at point \editb{for wich} $w = z_{n+1}\ne 0$, that 
	\begin{equation} 
		\ahlfors_{\alpha\bbar}(F)
		=
		D^{\|z\|^2 - 1}_{\alpha\bbar} \log Q_F
		=
		\phi_{\alpha}\delta_{\alpha\beta} + \zbar_{\alpha} z_{\beta} \,\phi_{\alpha\beta},
	\end{equation} 
	where
	\begin{align*}
		\phi_{\alpha} & = \frac{\partial \log g}{\partial t_{\alpha}}\bigl|_{t_j = |z_j|^2}, \quad \text{and}\\
		\phi_{\alpha\beta}
		& =
		\left(\frac{\partial^2 }{\partial t_{\alpha} t_{\beta}}
		-
		\frac{\partial^2 }{\partial t_{\beta} \partial t_{n+1}}
		-
		\frac{\partial^2 }{\partial t_{\alpha} \partial t_{n+1}}
		+
		\frac{\partial^2 }{\partial t_{n+1}^2}\right) \log g(t)\bigl|_{t_j = |z_j|^2},
	\end{align*}
	are real-valued. Since the inverse of \edit{the Levi matrix} is $h^{\alpha\bbar} = \delta_{\alpha\beta} - z_{\alpha}\, \zbar_{\beta}$, we obtain that $\ahlfors_{\alpha\bbar}(F)$ is tracefree on $|w|=1$. Moreover,
	\begin{equation} 
		\trace \ahlfors(F)
		=
		\sum_{\alpha} \phi_{\alpha} + 
		\sum_{\alpha} (\phi_{\alpha\alpha} - \phi_{\alpha}) |z_{\alpha}|^2 - \sum_{\alpha,\beta} \phi_{\alpha\beta} |z_{\alpha}|^2 \, |z_{\beta}|^2.
	\end{equation} 
	These prove the corollary.
\end{proof}

\subsection{The $(1,0)$-form $\ahlfors_1$ in dimension three}
As discussed in the last section, in dimension three, the $(1,0)$-form $\ahlfors_1$ \edit{can be thought of as} a replacement for the tracefree part of the Ahfors \edit{derivative}, which is trivial in \editb{in this case}.
%the \edit{three-dimension case}. 
It is possible to give an explicit formula for $\ahlfors_1$ in general, but we shall focus on the case where $M = \mathbb{S}^3$ is the \edit{3-sphere} with the standard pseudohermitian structure $\Theta$. Since the pseudohermitian torsion vanishes, we have $P(u) = u_{\bar{1},}{}^{\bar{1}}{}_{1}$ \edit{(covariant derivatives)}. Put 
\begin{equation}\label{e:standardframe}
	L = \wbar \partial_{z} - \zbar \partial_{w},
	\quad
	\theta^1 = w dz - z dw.
\end{equation}
Then clearly,
\begin{equation} 
	P(u) = (LL\overline{L}\, u)\, \theta^1.
\end{equation} 
In the case of the \edit{3-sphere}, the conjugate of $P$ reduces to the operator characterizing the CR pluriharmonic \editb{functions introduced} by Bedford \cite{bedford1974dirichlet}.

Let $F, G \colon \mathbb{S}^3 \to \mathbb{S}^N$ be CR maps with 
\begin{equation}
	\|F\|^2 - 1 = Q_F\cdot \rho, 
	\quad
	\|G\|^2 - 1 = Q_G \cdot \rho. 
\end{equation}
Then
\begin{equation} 
\ahlfors_1(F) 
= 
(LL\overline{L} \log Q_F)) \theta^1,
\end{equation} 
and similarly for $G$. Therefore, if $F = \phi \circ G$, then 
\begin{equation} 
	(LL\overline{L}\, \log Q_F) 
	=
	(LL\overline{L}\, \log Q_G)
\end{equation} 
Moreover, if $F \circ \gamma = \phi \circ G$, then
\begin{equation} 
(LL\overline{L}\, \log Q_F) \circ \gamma 
=
 e^{-2\psi}(LL\overline{L}\, \log Q_G)
\end{equation} 
where $\psi$ is determined by $\gamma^{\ast} \theta = e^{\psi}\theta$. 

{Similar to \cref{cor:monomial}, we have the following characteristic of the $\ahlfors_1$ of monomial maps. For a monomial map $F$ from \edit{a 3-sphere} into another sphere, $\ahlfors_1(F)$ expressed in the standard coframe \cref{e:standardframe}, must have the following form:
\begin{equation} 
	\ahlfors_1(F) = \bar{z} \bar{w} \editb{p}(|z|^2, |w|^2) / Q_F^3,
\end{equation} 
where $\editb{p}$ is a polynomial of $|z|^2$ and $|w|^2$ with real coefficients. 
In particular, the vanishing locus of $\ahlfors_1(F)$, which is an invariant for equivalent sphere maps, must \edit{contain} at least two circles. The proof of this fact is similar to that of the aforementioned corollary. We omit the \edit{details}.
\section{Examples and a question}
We calculate the Ahlfors tensor $\ahlfors$ and \edit{the $(1,0)$-form} $\ahlfors_1$ for various sphere maps that have \editb{previously appeared} in the literature. For the sphere case, the calculations are simple by \cref{prop:explicitahlfors}. All calculations can be done by hand, but \edit{some tedious calculations} can also be done by a computer algebra system.
\begin{example}\rm In \cite{d1988proper}, D'Angelo provided a list of 13 discrete examples and two 1-parameter analytic families of monomial maps from $\mathbb{S}^3$ to $\mathbb{S}^7$ which includes 4 trivial extensions of maps from $\mathbb{S}^3 \to \mathbb{S}^5$ of Faran's list. Later, Watanabe \cite{watanabe1992proper} found another map (numbered 16 in the \cref{t1}). We compute the trace of the Ahlfors derivative of each map. We then locate the umbilical points of the images of $\mathbb{S}^3$ in $\mathbb{S}^7$ via the maps. There are four types of umbilical loci that occur: the empty set, the whole sphere, one circle, \edit{and} the union of two circles; this is predicted in \cref{cor:monomial}. To simplify the notations, we put $S_1 = \{(e^{it},0) \colon t\in \mathbb{R}\}$ and $S_2 = \{(0,e^{it}) \colon t\in \mathbb{R}\}$. In \cref{t1}, the expressions in the Ahlfors column are the traces of the Ahlfors derivatives which only depend on $|z|^2$, and thus we put $s=|z|^2$. This trace determines the ``hermitian part'' \edit{(or mixed-part)} of the Ahlfors derivative as $n=1$. The ``holomorphic'' parts vanish in all \edit{cases} since the standard spheres have vanishing pseudohermitian torsions. Moreover, the norms of the CR second fundamental forms $|\sff^{CR}|$ can be computed easily from these results. 
	
Observe that the Ahlfors derivatives of these 16 maps are all different. Thus, two different maps are not left equivalent. Moreover, the umbilical loci can also be used to distinguish equivalent classes. For examples, the maps in the Faran's list are pairwise nonequivalent, since their umbilical loci are not congruent under the CR automorphisms of $\mathbb{S}^3$.

We point out that in \cref{t1}, the maps numbered 1, 3, and 5, are special cases of the homogeneous maps. The trace of \edit{the Ahlfors derivatives} are constant while the $\ahlfors_1$'s vanish identically. More generally, for homogeneous maps of degree $d$ from $\mathbb{S}^{2n+1}$ with $n\geq 1$, the Ahlfors derivatives are nonzero multiples of the Levi metric.

\begin{landscape}
\small
\renewcommand{\arraystretch}{1.5}
\begin{table}
\begin{tabular}{|l|l|l|l|l|} 
\hline
~ & $F$ & $\trace\ahlfors(F)\quad s = |z|^2$ &$\ahlfors_1(F)\quad s= |z|^2$& Umbilic \\
\hline
1 & $(z,w,0,0)$ & $0$ &0& $\mathbb{S}^3$ \\ %\hline 
2 & $(z,zw,w^2,0) $ & $\dfrac{s}{\left(s-2\right)^2}$ &$-\dfrac{4 \bar z \bar w}{\left(s-2\right)^3}$&  $S_2$ \\ %\hline 
3 & $(z^2,\sqrt{2}zw,w^2,0)$ & $ \dfrac12$ &$0$& $\emptyset$ \\ %\hline
4 & $(z^3,\sqrt{3}zw,w^3,0)$ & $-\dfrac{3
   s \left(s-1\right)}{\left(s^2 -s+1\right)^2}$ &$\dfrac{2 \bar z \bar w \left(s-2\right) \left(s+1\right) \left(2 s-1\right)}{\left(s^2-s+1\right)^3}$& $S_1 \cup S_2$ \\ %\hline
5 & $(z^3 , \sqrt{3} z^2w , \sqrt{3} zw^2 , w^3)$ & 1 &$0$& $\emptyset$ \\ %\hline 
6 & $(z^3,z^2w,zw,w)$ & $\dfrac{\left(1-s\right) \left(s^2 +4 s+1\right)}{\left(s^2+s+1\right)^2}$ &$-\dfrac{18 \bar z \bar w s \left(s+1\right)}{\left(s^2+s+1\right)^3}$& $S_2$ \\ %\hline
7 & $(z^2,z^2w,zw^2,w)$ & $\dfrac{s^3+4 s^2-5 s+2}{\left(-s^2+2 s+1\right)^2}$ &$-\dfrac{8 \bar z \bar w \left(s^3+3 s-2\right)}{\left(s^2-2 s-1\right)^3}$& $\emptyset$ \\ %\hline
8 & $(z^2, \sqrt{2}z^2w , \sqrt{2} zw^2, w^2)$ & $\dfrac{7 s^2-7 s+3}{2
   \left(-s^2+s+1\right)^2}$ &$-\dfrac{2 \bar z\bar w \left(2
   s-1\right) \left(s^2-s+4\right)}{\left(s^2-s-1\right)^3}$& $\emptyset$ \\ %\hline 
9 & $(z^3, \sqrt{3}z^2w , \sqrt{2}zw^2, w^2)$ & $\dfrac{9 s^2 -6 s+6}{\left(-s^2+2
   s+2\right)^2}$ &$-\dfrac{12 \bar z \bar w
   \left(s^3+6 s-4\right)}{\left(s^2-2 s-2\right)^3}$& $\emptyset$ \\ %\hline 
10 & $(z , z^2w , \sqrt{2}zw^2 , w^3)$ & $\dfrac{2
   s^2-3 s+3}{\left(3-2 s\right)^2}$ &$-\dfrac{12\bar z \bar w}{\left(2 s-3\right)^3}$& $\emptyset$ \\ %\hline
11 & $(z^4, z^3w , \sqrt{3}zw , w^3)$ & $-\dfrac{3 s \left(s^4-4 s^3+12 s^2-9\right)}{\left(s^3+3 s^2-3 s+3\right)^2}$ &$-\dfrac{36 \bar z \bar w
   \left(s^2-2 s-1\right) \left(s^4-2 s^3-6 s+3\right)}{\left(s^3+3 s^2-3
   s+3\right)^3}$& $S_1 \cup S_2$ \\ %\hline 
12 & $(z^4 , \sqrt{3}z^2w , zw^3 , w)$ & $-\dfrac{3
   \left(9 s^4-12 s^2 +4 s-1\right)}{\left(3 s^3-3 s^2+3 s+1\right)^2}$ &$\dfrac{36 \bar z \bar w \left(s^2+2 s-1\right) \left(3 s^4-6 s^3-2 s+1\right)}{\left(3 s^3-3 s^2+3 s+1\right)^3}$& $S_1$ \\ %\hline 
13 & $(z^5 , \sqrt{5}z^3w , \sqrt{5}zw^2 , w^5)$ & $ \dfrac{s \left(s^6-7 s^5+15
   s^4-25 s^3+35 s^2-27 s+9\right)}{\left(s^4-3 s^3+4
   s^2-2 s+1\right)^2}$ &$\mathcal{A}_1 (F)$& $S_2$ \\ % \hline 
14 & $(z , tw , \sqrt{1-t^2} zw, \sqrt{1-t^2} w^2)$ & $-\dfrac{\left(t^2-1\right) s}{\left(t^2 \left(1-s\right)+s-2\right)^2}$ &$\dfrac{2 \bar z \bar w \left(t^4-3
   t^2+2\right)}{\left(\left(t^2-1\right) s-t^2+2\right)^3}$& $S_2$ \\ %\hline 
15 & $(z^2, \sqrt{1+t^2} zw , tw^2 , \sqrt{1-t^2} w)$ & $\dfrac{t^2  \left(s+1\right)-s+1}{\left(t^2 \left(1-s\right)+s+1\right)^2}$ &$\dfrac{4\bar z \bar w(1-
   t^4)}{\left(\left(t^2-1\right) s-t^2-1\right)^3}$& $\emptyset$ \\ %\hline 
16 & $(z^2,\sqrt{2}zw,zw^2,w^3)$ & $\dfrac{-7 s^2 +6 s+3}{\left(s^2-2
   s+3\right)^2}$ &$\dfrac{8\bar z \bar w \left(s^3-9 s+6\right)}{\left(s^2-2 s+3\right)^3}$& $\emptyset$ \\ \hline 
\end{tabular}
\medskip

\caption{The \edit{CR Ahlfors derivative} $\ahlfors$ and \edit{the $(1,0)$-form} $\ahlfors_1$ of monomial sphere maps from $\mathbb{S}^3 \to \mathbb{S}^7$.}
\label{t1}
\end{table}
\end{landscape}
The formula for the $\ahlfors_1$ tensor of the $13^{\mathrm{rd}}$ map is too big to fit in the table. Precisely,

\begin{align*}
\ahlfors_1(F_{13}) = 2 \bar z \bar w \frac{2 s^9-6 s^8+18 s^7-64 s^6+117 s^5-114 s^4+42 s^3+24 s^2-27 s+6}{\left(s^4-3 s^3+4 s^2-2 s+1\right)^3}.
\end{align*}
\end{example}
\begin{example}\rm
	Consider the following cubic map (appeared earlier in \cite{d1991polynomial}) 
\begin{equation} 
	\left(\frac{z^2 - z^2w}{\sqrt{2}} , \frac{zw - zw^2}{\sqrt{2}} , \frac{z+zw}{\sqrt{2}} , w^2 \right).
\end{equation} 
Then the trace of the \edit{CR Ahlfors derivative} is
\begin{equation} 
	\trace \ahlfors(F)
	=
	\frac{|w|^6 + 4|w|^4 + 2 \Re(w)(1-3|w|^2) - 5|w|^2 + 4}{Q}\biggr|_{\mathbb{S}^3}
\end{equation} 
for some polynomial $Q$ positive on $\mathbb{S}^3$. Observe that this trace vanishes if and only if $w=1$, i.e., the umbilical locus is a \textit{singleton}. This immediately implies that $F$ is not equivalent to any map in \cref{t1} whose umbilical locus is either empty or of positive dimension. In fact, by \cref{cor:monomial}, it is not equivalent to any monomial map regardless of the target dimension, a fact that was first observed by D'Angelo in \cite{d1991polynomial} for the target dimension 4.
\end{example}
\begin{example}\rm
	Consider the following map which was discussed in \cite{faran2010rational}, Proposition~3.3,
	\begin{equation}
	F(z,w) = \left(\frac{\sqrt{3}}{9} (z^2+4z-2), \frac{\sqrt{6}}{9}(z^2+z+1) , \frac{\sqrt{3}}{12} w(3z+5) , \frac{\sqrt{6}}{6} w^2 , \frac{\sqrt{13}}{12} w(z-1) \right)
	\end{equation}  
	Observe that this is not monomial and does not send $0$ to $0$.

	Then the trace of \editb{its CR} Ahlfors \editb{derviative} is
	\begin{equation}
	\trace\ahlfors(F)
	=
	\frac{30|z|^2 + 24\Re z + 18}{|z|^4 - 16|z|^2\Re z + 32\Re(z^2) + 272\Re z + 289}.
	\end{equation}
	Since this trace does not vanish on $\mathbb{S}^3$, $F(\mathbb{S}^3)$ is a submanifold of $\mathbb{S}^9$ without umbilical point. Moreover,
	\begin{equation} 
		\ahlfors_1(F) 
		=
		\frac{264\, \wbar (1 + 4 \zbar + \zbar^2)}{(17 +  8\, \Re z - |z|^2)^3}.
	\end{equation} 
	\edit{Thus}, $F$ is not left equivalent to any monomial map from $\mathbb{S}^3$ into $\mathbb{S}^N$, $N\geq 3$.
\end{example}

\begin{example}[D'Angelo's maps \cite{d1988proper}]\rm For each $t$, put $c = \cos(t)$ and $s = \sin (t)$ and consider the maps
\begin{equation}
	F_t(z,w)
    =
    (z_1, \dots , z_n, cw, sz_1w, \dots , s z_n w, sw^2).
\end{equation}
Then $F_t$ maps $\mathbb{S}^{2n+1}$ into $\mathbb{S}^{4n+3}$. It is clear that $F_0 = L$ is the linear embedding and $F_{\pi/2} \cong  (\Wm, 0)$ where $\Wm$ is the complex Whitney map from $\mathbb{S}^{2n+1}$ \editb{into} \edit{$\mathbb{S}^{4n+1}$}. \editb{We compute that}
 
\begin{equation}
	Q_t = 1+ s^2|w|^2, 
	\quad 
	F_t^{\ast} \Theta = i\bar{\partial} \rho =  e^{u_t} \theta, \quad u_t = \log (1 + s^2 |w|^2).
\end{equation}

In \editb{the} local frame $\widetilde{Z}_{\alpha} = \wbar \partial_{\alpha} - \zbar_{\alpha} \partial_w$, we have that 
\begin{equation}
    \ahlfors_{\alpha\bbar}(F_{\editb{t}})
    =
    \frac{s^2 \zbar_{\alpha} z_{\beta}}{(1+s^2|w|^2)^2},
\end{equation}
is non-negative and of rank one except at the umbilical points. Taking the trace, we have
\begin{equation}
	\trace \ahlfors(F_{\editb{t}})
	=
	\frac{s^2(1-|w|^2)}{(1+s^2|w|^2)^2}\biggr|_{\mathbb{S}^3}.
\end{equation}
Thus, $F_t$ \editb{and $F_{t'}$} are not left equivalent for $t\neq t'$. Moreover, if $\Psi \in \aut(\mathbb{S}^{2n+1},F_t)$ for $t\ne 0$, then $\Psi(0,w) = (0, \psi_{n+1}(0,w))$. The inclusion $F(\mathbb{S}^{2n+1}) \subset \mathbb{S}^{4n+3}$ is umbilic along the locus $\{F(0, e^{iy})\}$. 
\end{example}

The following example appears originally in Webster \cite{webster1979rigidity}.
\begin{example} \label{ex:7}\rm
    Let $M\subset \mathbb{C}^{n}$ be the strictly pseudoconvex 
    real hypersurface defined by $\rho = 0$, with 
   \begin{equation}
        \rho = |z|^2 + b(z) + \overline{b(z)} -1.
   \end{equation}
    Put $\theta: = i\bar{\partial} \rho$ and let 
   \begin{equation}
        F(z) = \frac{1}{1-b(z)}\left(z_1 , \dots , z_n , b(z)\right).
   \end{equation}
    Then $F$ maps $M$ into the unit sphere in $\mathbb{C}^{n+1}$, \editb{actually, }one can compute $\|F\|^2 - 1 = |1-b(z)|^{-2} \rho.$ Since $D_{\alpha\bbar}^{\rho} \log |1 - b(z)|^{-2} = 0$, we see that 
    \begin{equation}
    	\tf \ahlfors_{\alpha\bbar}(F) 	=
    	0.
\end{equation}
\editb{Let} $L$ and $\mathcal{W}$ be the linear embedding and Whitney map from $\mathbb{S}^{n+1}$ into $\mathbb{S}^{2d+1}$ with $d\geq 2n-1$. Then $L \circ F$ and $\mathcal{W} \circ F$ are inequivalent CR immersions from $M$ into $\mathbb{S}^{2d+1}$. (In general, post composing with inequivalent maps may still yield equivalent maps. For example, take $g$ to be the linear embedding of $\mathbb{S}^{2n-1}$ into $\mathbb{S}^{2n+1}$, then $L\circ g = \mathcal{W}\circ g$ is the linear embedding.)

To compute the trace part, we need to compute the Fefferman determinant $J[\rho]$. In a special case $b(z) = \frac{1}{2} \sum_{k=1}^{n} z_k^2$, we can easily compute $J[\rho] = \rho + 2$ and thus,
\begin{equation}
	\ahlfors_{\alpha\bbar}(F) = (n+1)^{-1} h_{\alpha\bbar}.
\end{equation}
In particular, the imbedding image has no umbilical point.
\end{example}
\begin{question}\label{q:1}\rm
	Among the examples discussed above, we have encountered three families of CR maps 
	with the property that their Ahlfors derivatives are constant multiples of the Levi metric: linear embeddings between spheres, the homogeneous maps, and the ones considered in \cref{ex:7}. Motivated by these examples, we pose the following question: Suppose that $F$ is a nonconstant CR map between spheres and 
	that $\ahlfors(F)$ is a nonzero constant (resp. functional) multiple of the Levi metric (i.e., $\ahlfors_{\alpha\beta}(F) = 0$  or $\ahlfors_{\alpha\bbar}(F) = g h_{\alpha\bbar}$, $g \ne 0$, respectively). Does it \edit{follow} that $F$ is spherical\editb{ly} equivalent to a homogeneous monomial map?
\end{question}
%%%%%%%%%%%%%%%%%%%%%%

\end{document}